\documentclass[onefignum,onetabnum]{siamart220329}

\usepackage{lipsum}
\usepackage{amsfonts}
\usepackage{graphicx}
\usepackage{epstopdf}

\usepackage{bbm}

\usepackage{amsmath,amssymb}
\usepackage{bm}
\usepackage{thmtools}
\usepackage{thm-restate}

\usepackage[labelfont=bf,textfont=it,font=small]{caption}
\usepackage{subcaption}

\usepackage{algorithm}
\usepackage{algorithmicx}
\usepackage{algpseudocode}
\algrenewcommand\algorithmiccomment[1]{\hfill{\color{gray}$\triangleright$~#1}}
\algrenewcommand\algorithmicindent{1em}%
\makeatletter
\newcommand{\algmargin}{\the\ALG@thistlm}
\makeatother
\newlength{\whilewidth}
\algdef{SE}[parWHILE]{parWhile}{EndparWhile}[1]
  {\parbox[t]{\dimexpr\linewidth-\algmargin}{%
     \hangindent\whilewidth\strut\algorithmicwhile\ #1\ \algorithmicdo\strut}}{\algorithmicend\ \algorithmicwhile}%
\algnewcommand{\parState}[1]{\State%
  \parbox[t]{\dimexpr\linewidth-\algmargin}{\strut #1\strut}}

\usepackage{dsfont}

\usepackage{xcolor}
\definecolor{c0}{HTML}{641a80}
\definecolor{c1}{HTML}{b73779}
\definecolor{mplb}{HTML}{1f77b4}
\definecolor{mplo}{HTML}{ff7f0e}
\definecolor{mplg}{HTML}{2ca02c}

\usepackage{tikz}
\usetikzlibrary{graphs,patterns} 
\usepackage{adjustbox}

\usepackage{comment}

\crefformat{equation}{#2(#1)#3}
\crefmultiformat{equation}{{}#2(#1)#3}{ and~#2(#1)#3}{, #2(#1)#3}{ and~#2(#1)#3}
\Crefformat{equation}{#2(#1)#3}
\Crefmultiformat{equation}{{}#2(#1)#3}{ and~#2(#1)#3}{, #2(#1)#3}{ and~#2(#1)#3}

\crefname{assumption}{Assumption}{Assumptions}
\crefname{problem}{Problem}{Problems}
\crefname{section}{Section}{Sections}
\crefname{subsection}{Subsection}{Subsections}

\renewcommand{\vec}{\bm}

\newcommand{\llbracket}{[\![}
\newcommand{\rrbracket}{]\!]}

\newcommand{\HODLR}{\operatorname{HODLR}}
\newcommand{\HSS}{\operatorname{HSS}}
\newcommand{\SSS}{\operatorname{SSS}}
\newcommand{\UBLR}{\operatorname{BLR^2}}
\newcommand{\Gaussian}{\operatorname{Gaussian}}

\newcommand{\EE}{\operatorname{\mathbb{E}}}

\newcommand{\R}{\mathbb{R}}

\newcommand{\T}{\mathsf{T}}
\newcommand{\F}{\mathsf{F}}

\newcommand{\blockdiag}{\operatorname{blockdiag}}

\newcommand{\rank}{\operatorname{rank}}

\usepackage{nicematrix}

\newcommand{\di}[1]{{\color{orange}{DH: #1}}}
\newcommand{\Cam}[1]{{\color{blue}{Cam: #1}}}
\newcommand{\Tyler}[1]{{\color{teal}{\textbf{Tyler:}} #1}}

\newcommand{\rev}[1]{\textcolor{black}{#1}}

\newcommand\numberthis{\addtocounter{equation}{1}\tag{\theequation}}

\usepackage[color=blue!10,linecolor=blue,textsize=footnotesize]{todonotes}
\makeatletter%
\@mparswitchfalse%
\makeatother%
\normalmarginpar

\newsiamremark{remark}{Remark}
\newsiamremark{hypothesis}{Hypothesis}
\crefname{hypothesis}{Hypothesis}{Hypotheses}
\newsiamthm{claim}{Claim}

\headers{Quasi-optimal HSS matrix approximation}{Amsel, Chen, Duman Keles, Halikias, Musco, Musco, Persson}

\renewcommand\funding[1]{{\bfseries Funding:} #1}

\title{Quasi-optimal Hierarchically semi-separable matrix approximation\thanks{%
\funding{Diana Halikias was supported by NSF grant DGE-2139899. Cameron Musco was supported in part by NSF grants CFF-2046235 and CCF-2427363. }}}

\author{Noah Amsel\thanks{New York University (\{\href{mailto:noah.amsel@nyu.edu}{noah.amsel},\href{mailto:fd2135@nyu.edu}{fd2135},\href{mailto:cmusco@nyu.edu}{cmusco}, \href{mailto:dup210@nyu.edu}{dup210}\}@nyu.edu)}
\and Tyler Chen\thanks{JP Morgan Chase (\email{tyler.chen@jpmchase.com})}
\and
Feyza Duman Keles\footnotemark[2]
\and
Diana Halikias\thanks{Cornell University (\email{dh736@cornell.edu})}
\and
Cameron Musco\thanks{University of Massachusetts Amherst (\email{cmusco@cs.umass.edu})}
\and
Christopher Musco\footnotemark[2]
\and
 David Persson\footnotemark[2]~\thanks{Flatiron Institute (\email{dpersson@flatironinstitute.org})}
}
\usepackage{amsopn}
\usepackage{wrapfig}

\ifpdf
\hypersetup{
  pdftitle={Quasi-optimal Hierarchically semi-separable matrix approximation},
  pdfauthor={Noah Amsel, Tyler Chen, Feyza Duman Keles, Diana Halikias, Cameron Musco, Christopher Musco, David Persson}
}
\fi

\allowdisplaybreaks

\begin{document}

\maketitle

\begin{abstract} We present a randomized algorithm for producing a quasi-optimal hierarchically semi-separable (HSS) approximation to an $N\times N$ matrix $\bm{A}$ using only matrix-vector products with $\bm A$ and $\bm A^\T$. 
We prove that, using $O(k \log(N/k))$ matrix-vector products and ${O}(N k^2 \log(N/k))$ additional runtime, the algorithm  returns an HSS matrix $\bm{B}$ with rank-$k$ blocks whose expected Frobenius norm error $\mathbb{E}[\|\bm{A} - \bm{B}\|_\F^2]$ is at most $O(\log(N/k))$ times worse than the best possible approximation error by an HSS rank-$k$ matrix. In fact, the algorithm we analyze in a simple modification of an empirically effective method proposed by [Levitt \& Martinsson, SISC 2024].
As a stepping stone towards our main result, we prove two results that are of independent interest: a similar guarantee for a variant of the algorithm  which  accesses $\bm{A}$'s entries directly, and explicit error bounds for near-optimal subspace approximation using \textit{projection-cost-preserving sketches}.
To the best of our knowledge, our analysis constitutes the first polynomial-time quasi-optimality result for HSS matrix approximation, both in the explicit access model and the matrix-vector product query model.
\end{abstract}



\section{Introduction}
Recently, there has been a surge of interest in structured matrix approximation from matrix-vector products \cite{LiuXingGuo:2021,BakshiClarksonWoodruff:2022,BastonNakatsukasa:2022,HalikiasTownsend:2023,park_nakatsukasa_23,AmselChenDumanKelesHalikiasMuscoMusco:2024,ChenDumanKelesHalikiasMuscoMuscoPersson:2025,PearceYesypenkoLevittMartinsson:2025}.
Concretely, given an $N\times N$ matrix $\bm A$ and a class of structured matrices $\mathcal S$ (rank-$k$, sparse, butterfly, etc.), the goal is to find a matrix $\bm B \in \mathcal S$ approximating $\bm{A}$ in the sense that 
\begin{equation}\label{eqn:approx_target}
    \| \bm A - \bm B\|_\F \leq \Gamma \cdot \min_{\bm S \in \mathcal S} \|\bm A - \bm S \|_\F,
\end{equation}
for some approximation factor $\Gamma$.\footnote{If $\bm A$ is itself in $\mathcal{S}$, then this problem is referred to as \emph{matrix recovery}. Approximation guarantees for other norms are also of interest.} Moreover, we wish to do so given only \emph{implicit access} to $\vec{A}$ through black-box matrix-vector product queries $\bm{x}\mapsto \bm{A}\bm{x}$ and $\bm{x} \mapsto \bm{A}^\T\bm{x}$, both of which will henceforth be referred to as ``matvec queries to $\bm{A}$.''
Perhaps the most well-studied version of this problem is when $\mathcal{S}$ is the class of low-rank matrices, for which methods like (block) Krylov iteration and the randomized SVD can be used \cite{HalkoMartinssonTropp:2011, BakshiClarksonWoodruff:2022,BakshiNarayanan:2023,SimchowitzElAlaouiRecht:2018}. 
Recent work has also considered approximation algorithms for classes such as diagonal matrices \cite{BekasKokiopoulouSaad:2007,TangSaad:2011,BastonNakatsukasa:2022,DharangutteMusco:2023}, sparse matrices \cite{CurtisPowellReid:1974,ColemanMore:1983, ColemanCai:1986,WimalajeewaEldarVarshney:2013,DasarathyShahBhaskarNowak:2015,park_nakatsukasa_23,AmselChenDumanKelesHalikiasMuscoMusco:2024}, and beyond \cite{WatersSankaranarayananBaraniuk:2011,LiYang:2017,SchaferKatzfussOwhadi:2021,LiuXingGuo:2021, kapralov2023toeplitz}. 

\subsection{Hierarchical low-rank matrices} 
In the computational sciences, the class of matrices with \emph{hierarchical low-rank structure} is of particular interest \cite{Martinsson:2008,LinLuYing:2011}. Applications include solvers for differential and integral equations~\cite{BebendorfHackbusch:2003,XiaChandrasekaranGuLi:2010,XiaChandrasekaranGuLi:2010superfast,PouransariCoulierDarve:2017,Martinsson:2020, MasseiMazzaRobol:2019, KayeGolez:2021}, estimation of inverse covariance matrices~\cite{AmbisakaranForemanGreeengardHoggONeil:2016,Litvinenko2019}, Hessian approximation  \cite{Geoga2019,AmbartsumyanBoukaramBui-Thanh:2020}, and tasks in scientific machine learning such as operator learning \cite{BoulleTownsend:2022,BoulleHalikiasTownsend:2023}. 
Efficient data structures for hierarchical matrices enable many  linear algebraic tasks, such as computing matrix-vector products and solving linear systems, to be performed rapidly, often in time that is linear or nearly linear in the matrix dimension~\cite{BormReimer:2013,BallaniKressner:2016,MasseiRobolKressner:2020}.

As such, algorithms for computing hierarchical low-rank approximations are a central topic of research~\cite{Martinsson:2008, LinLuYing:2011, BoukaramTurkiyyahKeyes:2019, CaiHuangChowXi:2024,ChenDumanKelesHalikiasMuscoMuscoPersson:2025}. 
The matvec query model is especially natural for many applications of these methods. 
For instance, integral operators, which can  be accurately approximated by hierarchical matrices~\cite{BebendorfHackbusch:2003}, admit fast matrix-vector multiplication algorithms via the Fast Multipole Method~\cite{GreengardRokhlin:1987} or Barnes–Hut algorithm~\cite{BarnesHut:1986}. Matvec queries are also the natural access model in operator learning applications, where such queries correspond to data available from simulations or experiments~\cite{LuJinPangZhangKarniadakis:2021,  LiKovachkiLiuBhattacharyaStuartAnandkumar:2021, BoulleEarlsTownsend:2022,KovachkiLanthalerStuart:2024}. We refer the reader to \cite{BoulleTownsend:2024} for more details.

In this paper, we specifically study algorithms for solving \cref{eqn:approx_target} where $\mathcal{S}$ is the widely used class of \textit{hierarchically semi-separable} (HSS) matrices (see \Cref{def:HSS}).
Like the broader class of hierarchical off-diagonal low-rank (HODLR) matrices, HSS matrices are recursively partitioned into low-rank off-diagonal blocks and  on-diagonal blocks.
However, the column and row spaces of the low-rank blocks of HSS matrices at different levels of the hierarchy satisfy an additional nestedness property.
This results in a format that is even more compact than that of HODLR matrices.
At the same time, the added structure makes algorithms for HSS approximation substantially harder to design and analyze. Indeed, while there has been significant work on algorithms for obtaining HSS approximations (see \Cref{sec:prior_work} for details), even when the entries of $\vec{A}$ can be accessed explicitly, we are not aware of any polynomial-time algorithms that achieve a strong multiplicative error bound as in \cref{eqn:approx_target}. In the setting where $\vec{A}$ is only accessible through  matrix-vector products, we are only aware of one existing algorithm by  Levitt and Martinsson \cite{LevittMartinsson:2024}. However, while the algorithm works well in practice, as noted in~\cite[Remark 4.3]{LevittMartinsson:2024} it is not well-understood theoretically.

\subsection{Contributions}
Our main contribution is to show that a modification of the Levitt-Martinsson algorithm does in fact produce a quasi-optimal HSS approximation to an arbitrary $N\times N$ matrix $\bm{A}$. 
Specifically, in \Cref{theorem:HSS_matvec}, we show that, using $O(k\log(N/k))$ matrix-vector products with $\bm{A}$ and $O(Nk^2\log(N/k))$ additional time, we can produce an HSS rank-$k$ matrix $\bm{B}$ whose expected error. $\mathbb{E}[\|\bm{A} - \bm{B}\|_\F^2]$, is at most a $O(\log(N/k))$ multiplicative factor worse than the error of the best possible HSS rank-$k$ approximation error to $\bm{A}$. Via a simple application of Markov's inequality, our method thus solves \cref{eqn:approx_target} with probability $1-\delta$ for $\Gamma = O(\log(N/k)/\delta)$.



In~\Cref{section:nearoptimal}, we start by analyzing a general greedy approach for HSS approximation, which captures our variant of the Levitt-Martinson algorithm as a special case. This approach computes a sequence of approximations from the simpler sequentially semi-separable (SSS) matrix class. 
Roughly, we can compute one near-optimal SSS approximation for each level of the HSS hierarchy after fixing all lower levels, and argue that doing so leads to an HSS approximation that is near-optimal as well.
When implemented using explicit access to $\bm{A}$, this greedy algorithm is closely related to existing methods for HSS approximation, including those used in \texttt{hm-toolbox} \cite{MasseiRobolKressner:2020, XiaChandrasekaranGuLi:2010, yaniv2025construction}. 
We prove in \Cref{theorem:HSSgreedy} that such an approach yields a $2\log_2(N/k)$-approximation using $O(N^2k)$ arithmetic operations, which  
we believe is the first polynomial time result for relative error HSS-approximation in the explicit access model. 

In~\Cref{section:matvec}, we extend our analysis of the greedy algorithm to work in the implicit matvec query model. At a high-level, akin to the randomized SVD \cite{HalkoMartinssonTropp:2011}, $\vec{A}$ is multiplied by blocks of $O(k)$ random vectors (i.e., \emph{sketching matrices}) to recover rank-$k$ approximations to various blocks.   
We use the  \emph{block nullification} method from \cite{LevittMartinsson:2024} to simulate matvecs with multiple individual blocks of $\bm A$ given the result of a single random sketching matrix.
However,  unlike ~\cite{LevittMartinsson:2024},
we sample \emph{a fresh sketching matrices} for each level of the matrix's hierarchy. Doing so ensures the independence of certain matrices that arise in our error analysis, and is critical to proving theoretical approximation guarantees. 
A key component of our analysis is to analyze low-rank approximation using the \emph{projection-cost-preserving sketch} (PCPS) framework of \cite{CohenElderMusco:2015, musco2020projection}. PCPSs avoid the need for adaptive transpose products with individual blocks, which are required for more common low-rank approximation algorithms like the randomized SVD \cite{HalkoMartinssonTropp:2011}, but which are not easily simulated by block-nullification.
In our analysis, we establish new explicit error bounds for PCPSs that may be of independent interest.

In~\cref{sec:experiments}, we include numerical experiments to  illustrate the impact of using fresh sketches, and compare our algorithm with \cite{LevittMartinsson:2024}. While fresh sketches do improve accuracy, gains are relatively minor. Since \cite{LevittMartinsson:2024} uses only $O(k)$ matvecs, this leaves open the exciting possibility of removing the logarithmic factor from our $O(k\log(N/k)$ query complexity bound, while matching the same strong theoretical guarantees.

Finally, in \Cref{sec:BLR}, we generalize our results to the uniform block low-rank ($\UBLR$) class~\cite{AshcraftButtariMary:2021}, which has recently received attention in the matvec query model \cite{PearceYesypenkoLevittMartinsson:2025}. 
Our analysis for this class suggests how our techniques could be further  modified for matrices with arbitrary hierarchical partitions.

\subsection{Past work}
\label{sec:prior_work}
Most closely related to the present work is  \cite{ChenDumanKelesHalikiasMuscoMuscoPersson:2025}, which analyzes a matvec query algorithm for HODLR matrix approximation. 
The algorithm from \cite{ChenDumanKelesHalikiasMuscoMuscoPersson:2025} can be viewed as a robust version of widely used ``peeling'' methods for recovering matrices with exact HODLR structure from matrix-vector products \cite{LinLuYing:2011,Martinsson:2016,LevittMartinsson:2024a}. It obtains a rank-$k$ HODLR approximation $\vec{B}$ that is within a $(1+\varepsilon)$ factor of optimal using $O(k\log^4(N)/\varepsilon^3)$ matrix-vector products. Since the class of HSS matrices is a subset of HOLDR matrices, a simple triangle inequality argument implies that the best HSS rank-$k$ approximation to $\vec{B}$ is within a $(2+2\epsilon)$ factor of the optimal HSS approximation. In principal, this yields an approach to solving \cref{eqn:approx_target} but \rev{(i)} such an approach has worse matvec complexity than our method, which uses $O(k\log(N/k))$ queries and (ii) it is not computationally efficient, since it is not clear how to compute the best HSS approximation to the (explicit) HODLR matrix $\vec{B}$.

Also related to our work is research on algorithms for HSS approximation is the setting where we have explicit access to $\vec{A}$'s entries \cite{Borm:2010,MasseiRobolKressner:2020,kressnermasseirobol:2019, XiXiaCauleyBalakrishnan:2014}. 
Existing algorithms  approximate  low-rank submatrices up to a given fixed tolerance (e.g., using the truncated SVD).
As such, they provide additive error bounds with respect to this tolerance. 
For instance, \cite[Theorem 4.7]{kressnermasseirobol:2019} guarantees spectral norm error $O(\sqrt{N}\epsilon)$, where $\epsilon$ is the truncation tolerance for each low-rank block. As far as we are aware, such bounds do not translate to relative error bounds like those sought in \cref{eqn:approx_target}, which compare to the best possible HSS approximation of a given rank. 

There are considerably fewer results on HSS approximation in the matvec query model. 
In fact,  the only truly black-box matvec query algorithm we are aware of is that of \cite{LevittMartinsson:2024}, which builds on previous methods that use both matvec queries and a linear number of entry evaluations (i.e., are ``partially matrix-free'') ~\cite{Martinsson:2008,Martinsson:2011,gorman2019robust}. 
Efficient algorithms have also been developed for HSS subclasses, such as when $\vec A$ arises from a boundary integral operator in the plane~\cite{martinsson2005fast} and when $\vec A$ is exactly HSS with its largest off-diagonal blocks having rank exactly $k$~\cite{HalikiasTownsend:2023}.

\section{Background and Notation}
This section introduces necessary notation and provides definitions of the matrix classes (HSS and SSS) that we analyze. 

We denote the best rank-$k$ approximation to a matrix $\bm M$ in any unitarily invariant norm as $\llbracket\bm{M}\rrbracket_k$. 
This approximation can be obtained by the truncated SVD as described by the Eckart--Young Theorem.  
We write the transpose of $\bm{M}$  as $\bm M ^\T $ and the Moore-Penrose pseudoinverse as $\bm{M}^{\dagger}$. The Frobenius  and spectral norms are denoted by $\| \bm M \|_\F$ and $\|\bm{M}\|_2$, respectively. Throughout,  $\EE[\,\cdot\,]$ denotes expectation and $\text{Gaussian} (m, n)$ denotes an $m \times n$ matrix whose entries are independent standard normal random variables. 
We denote by $\blockdiag(\bm{B}_1,\ldots,\bm{B}_d)$ the block-diagonal matrix with (possibly rectangular) diagonal blocks $\bm{B}_1, \ldots, \bm{B}_d$.

\subsection{Notation for sub-matrices}
We  frequently interpret matrices as block matrices of varying sizes. The following definition introduces the notation that we use to reference individual blocks of different dimensions.

\begin{definition}\label{def:partitioning}
For non-negative integers $\ell,k$, consider a matrix $\bm{B} \in \mathbb{R}^{2^{\ell + 1} k \times 2^{\ell + 1} k}$. We use the following notation to write $\bm{B}$ as a $2^{\ell} \times 2^{\ell}$ block matrix
\begin{equation*}
    \bm{B} = \begin{bmatrix} 
        \bm{B}_{1,1} & \cdots & \bm{B}_{1,2^\ell} \\ 
        \vdots &  \ddots & \vdots \\
        \bm{B}_{2^{\ell},1} & \cdots & \bm{B}_{2^{\ell},2^\ell} \end{bmatrix}, \quad \text{where } \bm{B}_{i,j} \in \mathbb{R}^{2k \times 2k}.
\end{equation*}
We use alternative notation to write $\bm{B}$ as a $2^{\ell + 1} \times 2^{\ell + 1}$ block matrix
\begin{equation*}
    \bm{B} = \begin{bmatrix} 
        \widehat{\bm{B}}_{1,1} & \cdots & \widehat{\bm{B}}_{1,2^{\ell+1}} \\ 
        \vdots &  \ddots & \vdots \\
        \widehat{\bm{B}}_{2^{\ell+ 1},1} & \cdots & \widehat{\bm{B}}_{2^{\ell+ 1},2^{\ell+1}} \end{bmatrix}, \quad \text{where } \widehat{\bm{B}}_{i,j} \in \mathbb{R}^{k \times k}.
\end{equation*}
Note that we have the relationship
\begin{equation*}
    \bm{B}_{i,j} = \begin{bmatrix} \widehat{\bm{B}}_{2i-1,2j-1} & \widehat{\bm{B}}_{2i-1,2j}\\
    \widehat{\bm{B}}_{2i,2j-1} & \widehat{\bm{B}}_{2i,2j} \end{bmatrix}.
\end{equation*}
\end{definition}

Our analysis will repeatedly involve certain submatrices involving all off-diagonal blocks in a given block row or column.

\vfill
\begin{definition}\label{def:HSScolrow}
    Let  $\bm{B} \in \mathbb{R}^{2^{\ell + 1}k \times 2^{\ell + 1}k}$ be partitioned as in~\Cref{def:partitioning}. The $i^{\text{th}}$ HSS block-column and HSS block-row of $\bm{B}$ are respectively defined as 
    \begin{align*}
        c_i(\bm{B}) = 
        \begin{bmatrix} 
        \bm{B}_{1,i} \\ \vdots \\ \bm{B}_{i-1,i} \\ \bm{B}_{i+1,i} \\ \vdots \\\bm{B}_{2^{\ell},i}
        \end{bmatrix}
        \quad \text{and} \quad r_i(\bm{B}) =
        \begin{bmatrix} 
        \bm{B}_{i,1} & \cdots & \bm{B}_{i,i-1} & \bm{B}_{i,i+1} & \cdots & \bm{B}_{i,2^{\ell}} 
        \end{bmatrix}.
    \end{align*}
\end{definition}

\subsection{Rank-structured matrices}\label{section:HSS}

Following \cite{LevittMartinsson:2024}, we define HSS matrices using a so-called \emph{telescoping factorization}.\footnote{For a discussion on the equivalence between the telescoping factorization and the perhaps more common ``data-sparse'' definition of HSS matrices, we refer the reader to \cite[Section 3]{CasulliKressnerRobol:2024}.} 
\begin{definition}\label{def:HSS}
Consider a matrix $\bm{B}\in\mathbb{R}^{2^{L+1}k\times 2^{L+1}k}$.
We say that $\bm{B}\in \HSS(L,k)$ \rev{if there exist block-diagonal matrices 
\begin{equation*}
    \vec{U}^{(\ell)}, \vec{V}^{(\ell)} \text{ for } \ell = 1,\ldots, L, \quad \vec{D}^{(\ell)} \text{ for } \ell = 0,\ldots,L
\end{equation*}
such that}
\begin{align}
\begin{split}\label{eq:telescoping_factorization}
    \bm{B}^{(L+1)} &=\bm{B},\\
    \bm{B}^{(\ell+1)} &= \bm{U}^{(\ell)} \bm{B}^{(\ell)} (\bm{V}^{(\ell)})^\T + \bm{D}^{(\ell)} \quad \text{for } \ell = 1,\ldots,L, \quad \text{and}\\
    \bm{B}^{(1)} &= \bm{D}^{(0)},
    \end{split}
\end{align}
where $\bm{D}^{(0)} \in\mathbb{R}^{2k\times 2k}$ and, for $\ell = 1,\ldots,L$, 
\begin{align*}
    \bm{U}^{(\ell)} &= \blockdiag\left(\bm{U}_1^{(\ell)},\ldots,\bm{U}_{2^\ell}^{(\ell)}\right) \quad \text{where } \bm{U}_i^{(\ell)} \in \mathbb{R}^{2k\times k} \text{ has orthonormal columns,}\\
    \bm{V}^{(\ell)} &= \blockdiag\left(\bm{V}_1^{(\ell)},\ldots,\bm{V}_{2^\ell}^{(\ell)}\right)  \quad \text{where } \bm{V}_i^{(\ell)} \in \mathbb{R}^{2k\times k} \text{ has orthonormal columns,}\\
    \bm{D}^{(\ell)} &= \blockdiag\left(\bm{D}_1^{(\ell)},\ldots,\bm{D}_{2^\ell}^{(\ell)}\right) \quad \text{where } \bm{D}_i^{(\ell)} \in \mathbb{R}^{2k \times 2k}.
\end{align*}
\end{definition}

We note that the telescoping factorization in~\Cref{def:HSS} encodes the recursive nature of an HSS matrix. That is, if $\bm B = \bm B^{(L+1)}$ is an $N \times N$ matrix in $\HSS(L, k)$, then  $\bm B^{(L)}$ is an $N/2 \times N/2$ matrix in  $\HSS(L-1, k)$, and so on.
One level ($\ell=2$) of the telescoping decomposition is illustrated below.
\rev{The telescoping factorization also implies that for any $i\neq j$, the $(i,j)$ block of $\vec{B}^{(\ell+1)}$ has rank at most $k$, with range spanned by the columns of $\vec{U}^{(\ell)}_{i}$ and co-range spanned by the columns of $\vec{V}^{(\ell)}_j$, for all $\ell = 1,\ldots,L$.}
This immediately implies the HSS block-rows and block-columns are rank-$k$.

\begin{minipage}{\textwidth}\centering

\begin{tikzpicture}[scale=.25]
\def\K{4}
\def\Km{3}

\draw[fill=black!40] (0,0) rectangle ({2*\K},{2*\K});
\foreach \x in {1,...,\Km}{
    \draw[dotted,line width=.2px] (0,{2*\x}) -- ({2*\K},{2*\x});
    \draw[dotted,line width=.2px] ({2*\x},0) -- ({2*\x},{2*\K});
}

\node[] at ({2*\K+1.5},{\K+\K/2}) [] {=};

\begin{scope}[shift = {(2*\K+3,0)}]
\foreach \x in {1,...,\Km}{
    \draw[dotted,line width=.2px] (0,{2*\x}) -- (\K,{2*\x});
    \draw[dotted,line width=.2px] (\x,0) -- (\x,{2*\K});
}
\draw[] (0,0) rectangle (\K,{2*\K});
\foreach \x in {1,...,\K}{
    \draw[fill=black!40] (\x-1,{2*\K-2*(\x-1)}) rectangle (\x,{2*\K-2*\x});
}

\begin{scope}[shift = {(\K+1,\K)}]
\draw[fill=black!40] (0,0) rectangle (\K,\K);
\foreach \x in {1,...,\Km}{
    \draw[dotted,line width=.2px] (0,\x) -- (\K,\x);
    \draw[dotted,line width=.2px] (\x,0) -- (\x,\K);
}
\end{scope}

\begin{scope}[shift = {(2*\K+2,\K)}]
\foreach \x in {1,...,\Km}{
    \draw[dotted,line width=.2px] (0,\x) -- ({2*\K},\x);
    \draw[dotted,line width=.2px] ({2*\x},0) -- ({2*\x},\K);
}
\draw[] (0,0) rectangle ({2*\K},\K);
\foreach \x in {1,...,\K}{
    \draw[fill=black!40] ({2*\K-2*(\x-1)},\x-1) rectangle ({2*\K-2*\x,\x});
}
\end{scope}
\end{scope}

\node[] at ({6*\K+6.5},{\K+\K/2}) [] {+};

\begin{scope}[shift = {(6*\K+8,0)}]
\foreach \x in {1,...,\Km}{
    \draw[dotted,line width=.2px] (0,{2*\x}) -- ({2*\K},{2*\x});
    \draw[dotted,line width=.2px] ({2*\x},0) -- ({2*\x},{2*\K});
}
\draw[] (0,0) rectangle ({2*\K},{2*\K});
\foreach \x in {1,...,\K}{
    \draw[fill=black!40] ({2*(\x-1)},{2*\K-2*(\x-1)}) rectangle ({2*\x},{2*\K-2*\x});
}

\end{scope}

\begin{scope}[shift= {(0,2*\K+1.5)}]
\node[] at (4.5,0) [] {$\bm{B}^{(\ell+1)}$};
\node[] at ({2*\K+1.5},0) [] {=};
\node[] at (13.5,0) [] {$\bm{U}^{(\ell)}$};
\node[] at (18.5,0) [] {$\bm{B}^{(\ell)}$};
\node[] at (25,0) [] {$(\bm{V}^{(\ell)})^\T$};
\node[] at ({6*\K+6.5},0) [] {+};
\node[] at (36.5,0) [] {$\bm{D}^{(\ell)}$};
\end{scope}

\end{tikzpicture}
\vspace{1em}

\end{minipage}

\begin{remark}
\Cref{def:HSS} requires that $\vec B$ has dimension $N = 2^{L+1}k$, so $L = \log_2(N/k) - 1$. \rev{This is to simplify notation. More generally, HSS matrices are defined in terms of a binary partition tree, where each node in the tree corresponds to an index set. In \Cref{def:HSS}, we  assume that the binary tree is a perfect binary partition tree with leaf nodes containing $2k$ indices. This ensures that all basis matrices $\bm{U}_i^{(\ell)}$ and $\bm{V}_i^{(\ell)}$ have dimensions $2k \times k$, and the diagonal blocks $\bm{D}_i^{(\ell)}$ have dimensions $2k \times 2k$, independently of $i$ and $\ell$, which simplifies notation. However, our analysis easily extends to arbitrary $N$ and diagonal block size, corresponding to arbitrary partition trees.}
We discuss this and the relation to other hierarchical classes in  \Cref{sec:other_hierarchical}.
\end{remark}

Our algorithm for HSS approximation will build on an algorithm for \emph{sequentially semi-separable} (SSS) matrix approximation.
SSS matrices are the non-recursive version of HSS matrices and frequently appear in the context of numerical solutions of integral operators \cite{ChandrasekaranDewildeGuPalsvanderVeen:2002,AshcraftButtariMary:2021,gillman2012direct,martinsson2005fast}.

\begin{definition}\label{def:HSS_simple}
Consider a matrix $\bm{B} \in \mathbb{R}^{2^{L + 1}k \times 2^{L + 1}k}$.
We say $\bm{B} \in \SSS(L,k)$ if it can be written as 
$ \bm{B} = \bm{U} \bm{X} \bm{V}^\T + \bm{D},$ where $\bm{X}\in\mathbb{R}^{2^{L} k \times 2^{L}k}$, and
\begin{align*}
    \bm{U} &= \blockdiag\left(\bm{U}_1,\ldots,\bm{U}_{2^{L}}\right) \quad \text{where } \bm{U}_i \in \mathbb{R}^{2k \times k} \text{ has orthonormal columns,}\\
     \bm{V} &= \blockdiag\left(\bm{V}_1,\ldots,\bm{V}_{2^{L}}\right) \quad \text{where } \bm{V}_i \in \mathbb{R}^{2k\times k} \text{ has orthonormal columns, and}\\
    \bm{D} &= \blockdiag\left(\bm{D}_1,\ldots,\bm{D}_{2^{L}}\right) \quad \text{where } \bm{D}_i \in \mathbb{R}^{2k\times 2k}.
\end{align*}
\end{definition}
SSS matrices are a special case of the more general class of uniform block low-rank ($\UBLR$) matrices, which have appeared in recent works \cite{AshcraftButtariMary:2021,PearceYesypenkoLevittMartinsson:2025}.
We discuss how the algorithms in this paper can be extended to such matrices in \cref{sec:BLR}.

\section{Near-optimal greedy HSS approximation}\label{section:nearoptimal}
We begin by analyzing a generic greedy algorithm (\Cref{alg:HSS_greedy_meta}) for computing an HSS approximation of an arbitrary matrix. This algorithm is closely related to many  existing algorithms \cite{LevittMartinsson:2024,Martinsson:2008,Martinsson:2011,XiaChandrasekaranGuLi:2010,XiaChandrasekaranGuLi:2010superfast}.  
It builds a telescoping factorization of the HSS approximation as in Definition \ref{def:HSS}, starting at  level $L$ and proceeding until level 1. At each level, the algorithm solves an SSS approximation problem (see \Cref{def:HSS_simple}) to obtain $\vec A^{(\ell)}$ from $\vec A^{(\ell+1)}$. It does so by choosing the factors $\bm{U}_i^{(\ell)}$ and $\bm{V}_i^{(\ell)}$ to be the best (or near-best) subspaces for the corresponding block-rows and block-columns of $\vec A^{(\ell+1)}$ \emph{independently} of one another.
$\vec D_i^{(\ell)}$ is chosen as a near-optimal diagonal block given these factors. We argue that this approach yields a quasi-optimal HSS approximation with quasi-optimality constant $O(L)$ where $L = \log_2(N/k)-1$ is the number of levels.

\begin{algorithm}
\caption{Greedy HSS approximation (meta-algorithm)}
\label{alg:HSS_greedy_meta}
\textbf{input:} Matrix $\bm{A} \in \mathbb{R}^{N \times N}$ where $N=2^{(L+1)}k$, rank parameter $k$.\\
\textbf{output:} An HSS matrix $\bm{B} \in \HSS(L,k)$
\begin{algorithmic}[1]
\State Set $\bm{A}^{(L+1)} = \bm{A}$
\For{$\ell = L, \ldots, 1$}
\State \rev{Partition $\bm{A}^{(\ell+1)} = \{ \bm{A}^{(\ell+1)}_{i,j} \}_{i,j}^{2^\ell}$ into blocks of size $\bm{A}_{i,j}^{(\ell+1)} \in \mathbb{R}^{2k \times 2k}$.}
\For{$i = 1,2,\ldots,2^\ell$}\label{line:begininner}
    \State Compute $\bm{U}_i^{(\ell)}\in\mathbb{R}^{2k\times k}$ with orthonormal columns approximately minimizing 
    \[
    \big\|r_i(\bm{A}^{(\ell+1)}) - \bm{U}_i^{(\ell)} (\bm{U}_i^{(\ell)})^\T r_i(\bm{A}^{(\ell+1)})\big\|_\F
    \] 
    \State Compute $\bm{V}_i^{(\ell)}\in\mathbb{R}^{2k\times k}$ with orthonormal columns approximately minimizing
    \[
    \big\| c_i(\bm{A}^{(\ell+1)}) - c_i(\bm{A}^{(\ell+1)}) \bm{V}_i^{(\ell)} (\bm{V}_i^{(\ell)})^\T \big\|_\F
    \]
    \State Compute $\bm{D}_i^{(\ell)}\in\mathbb{R}^{2k\times 2k}$ approximately minimizing 
    \[
    \big\| (\bm{A}_{i,i}^{(\ell+1)} - \bm{D}_i^{(\ell)}) - \bm{U}_i^{(\ell)} (\bm{U}_i^{(\ell)})^\T (\bm{A}_{i,i}^{(\ell+1)} - \bm{D}_{i}^{(\ell)})\bm{V}_i^{(\ell)} (\bm{V}_i^{(L)})^\T  \big\|_\F
    \]\label{line:diagonal}
\EndFor\label{line:endinner}
\State Set $\bm{A}^{(\ell)} = (\bm{U}^{(\ell)})^\T \left(\bm{A}^{(\ell+1)} - \bm{D}^{(\ell)}\right)\bm{V}^{(\ell)}$ \label{line:Aell}
\EndFor
\State $\bm{D}^{(0)} = \bm{A}^{(1)}$
\State \Return $\vec{B}$, a telescoping factorization defined as in \Cref{eq:telescoping_factorization}: \[\vec{B} = \vec{U}^{(L)}\left(\cdots\left(\vec{U}^{(1)}\vec{D}^{(0)}{\vec{V}^{(1)}}^\T + \vec{D}^{(1)}\right)\cdots\right){\vec{V}^{(L)}}^\T + \vec{D}^{(L)}\]
\end{algorithmic}
\end{algorithm}

\begin{restatable}{theorem}{HSSgreedy}\label{theorem:HSSgreedy}
    Let $\bm{B} \in \HSS(L,k)$ be the HSS approximation to $\bm A \in \R^{N \times N}$ output by~\Cref{alg:HSS_greedy_meta}.
    Suppose at each level $\ell = L,\ldots, 1$ of \Cref{alg:HSS_greedy_meta}, for all $i=1, \ldots, 2^\ell$, $\vec{U}^{(\ell)}_i$ and $\vec{V}^{(\ell)}_i$ satisfy
    {\small\begin{align*}
    \EE\Bigl[\|r_i(\bm{A}^{(\ell+1)}) - \bm{U}_i^{(\ell)} (\bm{U}_i^{(\ell)})^\T r_i(\bm{A}^{(\ell+1)})\|_\F^2 \Big| \vec{A}^{(\ell+1)} \Big] 
    &\leq \Gamma_{\textup{r}} \cdot \|r_i(\bm{A}^{(\ell+1)}) - \llbracket r_i(\bm{A}^{(\ell+1)}) \rrbracket_k\|_\F^2,
    \\
    \EE\Bigl[\| c_i(\bm{A}^{(\ell+1)}) - c_i(\bm{A}^{(\ell+1)}) \bm{V}_i^{(\ell)} (\bm{V}_i^{(\ell)})^\T\|_\F^2 \Big| \vec{A}^{(\ell+1)} \Big] 
    &\leq \Gamma_{\textup{c}} \cdot \|c_i(\bm{A}^{(\ell+1)}) - \llbracket c_i(\bm{A}^{(\ell+1)}) \rrbracket_k\|_\F^2.
    \end{align*}}
    and the diagonal factors $\vec{D}^{(\ell)}_i$ satisfy
    {\small
    \begin{align*}
        &\EE\Bigl[\| (\bm{A}_{i,i}^{(\ell+1)} - \bm{D}_i^{(\ell)}) - \bm{U}_i^{(\ell)} (\bm{U}_i^{(\ell)})^\T (\bm{A}_{i,i}^{(\ell+1)} - \bm{D}_{i}^{(\ell)})\bm{V}_i^{(\ell)} (\bm{V}_i^{(\ell)})^\T \|_\F^2   \Big| \vec{A}^{(\ell+1)}, \bm{U}_i^{(\ell)},\bm{V}_i^{(\ell)} \Big]
        \\&\hspace{.8em}\leq \Gamma_{\textup{d}} \cdot \Big[ \|r_i(\bm{A}^{(\ell+1)}) - \bm{U}_i^{(\ell)} (\bm{U}_i^{(\ell)})^\T r_i(\bm{A}^{(\ell+1)})\|_\F^2 
        + \| c_i(\bm{A}^{(\ell+1)}) - c_i(\bm{A}^{(\ell+1)}) \bm{V}_i^{(\ell)} (\bm{V}_i^{(\ell)})^\T\|_\F^2 \Big].
    \end{align*}
    }
    Then,\footnote{We are not aware of any \rev{prior} result guaranteeing the existence of an optimal HSS approximation. \rev{We therefore include a proof that an optimal HSS approximation always exists; see \Cref{appendix:existence}.}
    }
    \begin{equation}\label{eq:opt_bound}
        \mathbb{E}\Bigl[\|\bm{A} - \bm{B}\|_\F^2 \Big] \leq \left(\Gamma_{\textup{r}} +\Gamma_{\textup{c}} \right)(1+\Gamma_{\textup{d}})L\cdot \rev{\min\limits_{\bm{C} \in \HSS(L,k)}}\|\bm{A}-\bm{C}\|_\F^2.
    \end{equation}
\end{restatable}

As discussed, \Cref{alg:HSS_greedy_meta} can be implemented in many different ways \cite{LevittMartinsson:2024,Martinsson:2008,Martinsson:2011,XiaChandrasekaranGuLi:2010,XiaChandrasekaranGuLi:2010superfast}. We highlight some important examples below:

\vspace{-.3em}
\noindent\textbf{Explicit matrices:}\label{page:explicit}
If $\bm{A}$ is dense and stored in memory, the simplest way to implement \Cref{alg:HSS_greedy_meta} is to use the truncated SVD: \rev{explicitly form $\bm{A}^{(\ell + 1)}$ and} let $\bm{U}_i^{(\ell)}$ and $\bm{V}_i^{(\ell)}$ be the top $k$ left and right singular vectors of $r_i(\bm{A}^{(\ell+1)})$ and $ c_i(\bm{A}^{(\ell+1)})$, respectively, and let $\bm{D}_{i,i}^{(\ell)} = \rev{\bm{A}_{i,i}^{(\ell+1)}}$. This implementation requires $O(N^2k)$ operations.\footnote{\rev{Each iteration of the inner loop in lines~\ref{line:begininner}-\ref{line:endinner} requires computing the SVD of two matrices of sizes $2k \times (2^{\ell+1} k - 2k)$ and  $(2^{\ell+1} k - 2k) \times 2k$. Hence, at level $\ell$ the total cost of the inner loop is $O\left((2^{\ell+1})^2 k^3\right)$. Computing $\bm{A}^{(\ell)}$ in line~\ref{line:Aell} also requires $O\left((2^{\ell+1})^2 k^3\right)$ operations due to the block-diagonal structure of $\bm{U}^{(\ell)}, \bm{V}^{(\ell)},$ and $\bm{D}^{(\ell)}$. Thus, each level $\ell$ requires $O\left((2^{\ell+1})^2 k^3\right)$ operations. Summing over $\ell = L,\ldots,1$ and using $N = 2^{L+1}k$ gives an overall computational cost of \Cref{alg:HSS_greedy_meta} of $O(4^{L+1}k^3) = O(N^2k)$ operations. } }
Since at each level we compute the \emph{optimal} subspaces of each block-row and block-column, in the context of \Cref{theorem:HSSgreedy}, this yields $\Gamma_{\textup{r}} = \Gamma_{\textup{c}} = 1$ and $\Gamma_{\textup{d}} = 0$. Consequently, for this implementation of \Cref{alg:HSS_greedy_meta}, the output satisfies $
    \|\bm{A} - \bm{B}\|_\F^2 \leq 2L\rev{\min\limits_{\bm{C} \in \HSS(L,k)}} \|\bm{A} - \bm{C}\|_\F^2.$

\vspace{-.3em}
\noindent\textbf{Implicit matrices:}
When $\vec{A}$ can only be accessed using matvecs,  we leverage the ``block-nullification'' idea of \cite{LevittMartinsson:2024} to implement \Cref{alg:HSS_greedy_meta}.
This approach is described in \cref{section:matvec}, where we prove in~\Cref{theorem:HSS_matvec} that the output of this algorithm satisfies $
    \mathbb{E}\|\bm{A} - \bm{B}\|_\F^2 \leq O(L)\rev{\min\limits_{\bm{C} \in \HSS(L,k)}} \|\bm{A} - \bm{C}\|_\F^2,$ yielding our main result.

\vspace{-.3em}
\noindent\textbf{Structured matrices in memory:}
If $\vec{A}$ is sparse or itself hierarchical (but perhaps with a higher rank) and stored in memory, this may be leveraged to efficiently implement \Cref{alg:HSS_greedy_meta}; see e.g. \cite{XiaChandrasekaranGuLi:2010,XiaChandrasekaranGuLi:2010superfast} which compress an HSS rank-$2k$ matrix to an HSS rank-$k$ matrix.\footnote{This commonly arises since the sum of two HSS rank-$k$ matrices is an HSS rank-$2k$ matrix.}
The approximation bound will depend on the exact algorithm -- we do not focus on this case in this paper.

Our analysis of \Cref{alg:HSS_greedy_meta} relies on two key facts. First, in~\Cref{section:simpleoptimal} we show that obtaining the SSS approximation at level $\ell$ by optimizing  $\bm{U}_i^{(\ell)}$ and $\bm{V}_i^{(\ell)}$  independently and inexactly does not incur much more error than optimizing them simultaneously and exactly (\Cref{thm:simple_error}).
Second, in~\Cref{sec:HSS_level_optimality}, we show that the error of the best HSS approximation to $\bm{A}^{(\ell+1)}$ is bounded below by the error of the best HSS approximation to $\bm{A}^{(\ell)}$ (\Cref{lemma:HSS_optimal}). Together, these facts ensure that the error arising from the low-rank approximation at each of the $L$ levels is bounded by the overall optimum error, resulting in our $O(L)$ quasi-optimality factor. 
\rev{\begin{remark}
   One can also prove a deterministic version of~\cref{theorem:HSSgreedy}. In that case, the assumptions would be deterministic bounds (rather than in expectation) on the quality of our approximations to the right and left subspaces of $r_i(\bm{A}^{(\ell+1)})$ and $c_i(\bm{A}^{(\ell + 1)})$ at each level $\ell$, and the expectation in~\cref{eq:opt_bound} would be removed. The statement and proof is otherwise identical. However, we emphasize that our main result is the stated probabilistic version for several reasons. First, hierarchical matrix algorithms generally work level-by-level, and~\cref{theorem:HSSgreedy} reflects that  an approximation at one level can be correlated to those at all prior levels. This dependence between levels is one of the key difficulties in analyzing hierarchical matrix techniques. Secondly, we extend this result to the matvec access model. Matvec algorithms for hierarchical matrix approximation must be probabilistic in nature because they rely on sketching for low-rank approximation, which cannot be done deterministically. 
\end{remark}}

\subsection{Approximation of SSS matrices}\label{section:simpleoptimal}

In this section, we argue that the SSS approximation computed at each iteration of \Cref{alg:HSS_greedy_meta} is quasi-optimal. 
We extend the analysis in this section to the more general class of \emph{uniform block low-rank} ($\UBLR$) matrices \cite{AshcraftButtariMary:2021,PearceYesypenkoLevittMartinsson:2025} in \Cref{sec:BLR}.

\begin{theorem}\label{thm:simple_error}
    Consider a matrix $\bm{A} \in \mathbb{R}^{2^{\ell + 1}k \times 2^{\ell+1} k}$ and factors $\bm{U}, \bm{V}, \bm{D}$ as partitioned as in \Cref{def:partitioning}. Suppose that for all $1 \leq i \leq 2^\ell$,  $\bm{U}_i$ and $\bm{V}_i$ satisfy
    \begin{align*}
        \EE\Bigl[ \|r_i(\bm{A}) - \bm{U}_i \bm{U}_i^\T r_i(\bm{A})\|_\F^2 \Big]
        &\leq \Gamma_{\textup{r}} \cdot \|r_i(\bm{A}) - \llbracket r_i(\bm{A}) \rrbracket_k\|_\F^2, 
        \\
        \EE\Bigl[ \| c_j(\bm{A}) - c_j(\bm{A}) \bm{V}_j \bm{V}_j^\T\|_\F^2 \Big]
        &\leq \Gamma_{\textup{c}} \cdot \|c_j(\bm{A}) - \llbracket c_j(\bm{A}) \rrbracket_k\|_\F^2,
    \end{align*}
    and  $\bm{D}_i \in \mathbb{R}^{2k \times 2k}$ satisfies  
    \begin{align}
    \begin{split}
        &
        \EE\Bigl[ \| (\bm{A}_{i,i} - \bm{D}_i) - \bm{U}_i \bm{U}_i^\T (\bm{A}_{i,i} - \bm{D}_{i})\bm{V}_i \bm{V}_i^\T \|_\F^2 \Big| \bm{U}_i,\bm{V}_i \Big]
        \\&\hspace{5em}
        \leq \Gamma_{\textup{d}} \cdot \left(\|r_i(\bm{A}) - \bm{U}_i \bm{U}_i^\T r_i(\bm{A})\|_\F^2 +  \| c_j(\bm{A}) - c_j(\bm{A}) \bm{V}_j \bm{V}_j^\T\|_\F^2 \right).
        \end{split}\label{eq:diagonal}
    \end{align}

    Then, if $\bm{X} = \bm{U}^T(\bm{A} - \bm{D}) \bm{V}$ we have
    \begin{equation}
    \EE\Bigl[ \| \bm{A} - (\bm{U}\bm{X}\bm{V}^\T + \bm{D}) \|_\F^2 \Bigr] \leq (\Gamma_{\textup{r}}+\Gamma_{\textup{c}})(1+\Gamma_{\textup{d}}) \cdot \rev{\min\limits_{\bm{B} \in \SSS(\ell,k)}}\|\bm{A}-\bm{B}\|_\F^2.
    \end{equation}
\end{theorem}

\begin{proof}
Let $\widetilde{\bm{U}}, \widetilde{\bm{V}},$ $\widetilde{\bm{D}}$, and $\widetilde{\bm{X}} = \{ \widetilde{\bm{X}}_{i,j}\}_{i,j=1}^{2^\ell}$ be any optimal factors, i.e. satisfying
\begin{equation*}
     \|\bm{A} - (\widetilde{\bm{U}}\widetilde{\bm{X}}\widetilde{\bm{V}}^\T + \widetilde{\bm{D}})\|_{\F}^2 = \rev{\min\limits_{\bm{B} \in \SSS(\ell,k)}}\|\bm{A}-\bm{B}\|_\F^2 =: \text{OPT}^2.
\end{equation*}
Such factors are guaranteed to exist, since the error between $\bm{A}$ and $\widetilde{\bm{U}} \widetilde{\bm{X}} \widetilde{\bm{V}}^\T + \widetilde{\bm{D}}$ is given by 
\begin{align}
\begin{split}
    \|\bm{A} - (\widetilde{\bm{U}} \widetilde{\bm{X}} \widetilde{\bm{V}}^\T + \widetilde{\bm{D}})\|_\F^2 &=  \sum\limits_{i=1}^{2^\ell}\|\bm{A}_{i,i} - \widetilde{\bm{U}}_i\widetilde{\bm{X}}_{i,i} \widetilde{\bm{V}}_i^\T - \widetilde{\bm{D}}_i\|_\F^2 \\
    &\hspace{5em}+ \sum\limits_{i,j = 1, i\neq j}^{2^\ell} \|\bm{A}_{i,j} - \widetilde{\bm{U}}_i\widetilde{\bm{X}}_{i,j} \widetilde{\bm{V}}_j^\T\|_\F^2.
    \end{split}\label{rem:simple_diagonal}
\end{align}
For $\widetilde{\bm{U}}, \widetilde{\bm{V}},$ $\widetilde{\bm{D}}$, and $\widetilde{\bm{X}}$ to be optimal, we must have $\bm{A}_{i,i} = \widetilde{\bm{U}}_i\widetilde{\bm{X}}_{i,i} \widetilde{\bm{V}}_i^\T + \widetilde{\bm{D}}_i$, since we can choose $\bm{D}_i = \bm{A}_{i,i}$ and $\widetilde{\bm{X}}_{i,i} = \bm{0}$ so the first error term vanishes, while leaving the second error term unchanged. 
$\widetilde{\bm{X}}$ can then be optimally chosen as $\widetilde{\bm{U}}^\T(\bm{A} - \widetilde{\bm{D}}) \widetilde{\bm{V}}$, and the blocks of  $\widetilde{\bm{U}}$ and $\widetilde{\bm{V}}$ can be obtained by optimizing over a Cartesian product of Stiefel manifolds, which is a compact set. 

Now, as in \cref{rem:simple_diagonal}, the error between $\bm{A}$ and $\bm{U} \bm{X} \bm{V}^\T + \bm{D}$ is given by 
\begin{align}
    \|\bm{A} - (\bm{U} \bm{X} \bm{V}^\T + \bm{D}) \|_\F^2 
    &= \sum\limits_{i = 1}^{2^\ell} \| (\bm{A}_{i,i} - \bm{D}_{i}) - \bm{U}_i \bm{X}_{i,i} \bm{V}_i^\T \|_\F^2 
    \nonumber\\&\hspace{5em}+ \sum\limits_{i,j = 1, i \neq j}^{2^\ell} \|\bm{A}_{i,j} - \bm{U}_i \bm{U}_i^{\T} \bm{A}_{i,j} \bm{V}_j \bm{V}_j^\T \|_\F^2.
        \label{eqn:simple_errdef}
\end{align}
Since $\bm{U}_i \bm{X}_{i,i} \bm{V}_i^\T = \bm{U}_i \bm{U}_i^\T (\bm{A}_{i,i} - \bm{D}_{i})\bm{V}_i \bm{V}_i^\T$, we bound the expectation of the first term in the sum by assumption:     \begin{align}
        &
        \EE\Bigl[\| (\bm{A}_{i,i} - \bm{D}_i) - \bm{U}_i \bm{X}_{i,i} \bm{V}_i^\T \|_\F^2 \Big| \bm{U}_i, \bm{V}_i \Bigr]
        \nonumber\\&\hspace{5em}\leq \Gamma_{\textup{d}} \cdot \left(\|r_i(\bm{A}) - \bm{U}_i \bm{U}_i^\T r_i(\bm{A})\|_\F^2 +  \| c_j(\bm{A}) - c_j(\bm{A}) \bm{V}_j \bm{V}_j^\T\|_\F^2 \right).
        \label{eqn:diag_Gamma_d}
    \end{align}
    
Now, we define
\begin{equation*}
    \bm{R}_i = \begin{bmatrix} \bm{A}_{i,1} \bm{V}_1 \bm{V}_1^\T & \cdots & \bm{A}_{i,i-1} \bm{V}_{i-1} \bm{V}_{i-1}^\T & \bm{A}_{i,i+1} \bm{V}_{i+1} \bm{V}_{i+1}^\T & \cdots & \bm{A}_{i,2^{\ell}} \bm{V}_{2^{\ell}} \bm{V}_{2^{\ell}}^\T \end{bmatrix}.
\end{equation*}
Then, for each $i$, we can rewrite the second term in~\Cref{eqn:simple_errdef} as
\begin{align*}
    \sum_{j=1,j\neq i}^{2^\ell}
    \|\bm{A}_{i,j}^{(\ell)} - \bm{U}_i \bm{U}_i^\T \bm{A}_{i,j}^{(\ell)} \bm{V}_j \bm{V}_j^\T \|_\F^2
    &= \|r_i(\bm{A}) - \bm{U}_i \bm{U}_i^\T \bm{R}_i\|_\F^2
    \\= \|r_i(\bm{A})& - \bm{U}_i \bm{U}_i^\T r_i(\bm{A})\|_\F^2 + \|\bm{U}_i \bm{U}_i^\T r_i(\bm{A}) - \bm{U}_i \bm{U}_i^\T \bm{R}_i\|_\F^2,
    \numberthis\label{eqn:simple_jsum}
\end{align*}
where the second equality is due to the Pythagorean theorem. 

Inserting \Cref{eqn:simple_jsum,eqn:diag_Gamma_d} into \Cref{eqn:simple_errdef} and using the law of total expectation, we find
\begin{align}\label{eq:r+c}
&\EE\Bigl[ \|\bm{A} - (\bm{U} \bm{X} \bm{V}^\T + \bm{D}) \|_\F^2 \Bigr] 
\leq R+C, \text{ where}\\
    R&= \sum\limits_{i=1}^{2^{\ell}} (1+\Gamma_{\textup{d}}) \cdot \EE\Bigl[ \|r_i(\bm{A}) - \bm{U}_i \bm{U}_i^\T r_i(\bm{A})\|_\F^2 \Big],\nonumber
    \\C&=\sum\limits_{i=1}^{2^{\ell}}  \EE\Bigl[ \Gamma_{\textup{d}} \cdot\| c_j(\bm{A}) - c_j(\bm{A}) \bm{V}_j \bm{V}_j^\T\|_\F^2 
    + \|\bm{U}_i \bm{U}_i^\T r_i(\bm{A}) - \bm{U}_i \bm{U}_i^\T \bm{R}_i\|_\F^2 \Big].\nonumber
\end{align}

We begin by bounding $R$. Define 
$
\widetilde{\bm{X}}_{i,:} = \begin{bmatrix}
    \widetilde{\bm{X}}_{i,1}
    &    \widetilde{\bm{X}}_{i,2}
    &\cdots &     \widetilde{\bm{X}}_{i,2^L}
\end{bmatrix}$
Using our assumption on $\bm{U}_i$ and that $\llbracket r_i(\bm{A}) \rrbracket_k$ is the best rank-$k$ approximation to $r_i(\bm{A})$, we have $
\EE\Bigl[ \|r_i(\bm{A}) - \bm{U}_i \bm{U}_i^\T r_i(\bm{A})\|_\F^2 \Big]
\leq \Gamma_{\textup{r}} \cdot \|r_i(\bm{A}) - \llbracket r_i(\bm{A}) \rrbracket_k \|_\F^2$.
Then, since $\widetilde{\bm{U}}_{i} \widetilde{\bm{X}}_{i,:} \widetilde{\bm{V}}^\T$ is some rank-$k$ approximation to $r_i(\bm{A})$, and recalling that $\widetilde{\bm{X}}_{i,i} = \bm{0}$, 
\begin{align*}
    R &\leq \Gamma_{\textup{r}} (1+\Gamma_{\textup{d}}) \cdot \rev{\sum_{i=1}^{2^{\ell}} \|r_i(\bm{A}) - \llbracket r_i(\bm{A}) \rrbracket_k \|_\F^2}
    \\&\leq 
     \Gamma_{\textup{r}} (1+\Gamma_{\textup{d}})\cdot \sum\limits_{i=1}^{2^{\ell}}\|r_i(\bm{A}) - \widetilde{\bm{U}}_{i} \widetilde{\bm{X}}_{i,:} \widetilde{\bm{V}}^\T\|_\F^2 
    =  \Gamma_{\textup{r}} (1+\Gamma_{\textup{d}})\cdot  \text{OPT}^2.
\end{align*}

Next, we bound $C$. 
Using that $\bm{U}_i\bm{U}_i^\T$ is an orthogonal  projector, we have that  
\begin{align*}
    \sum\limits_{i=1}^{2^{\ell}} \|\bm{U}_i \bm{U}_i^\T r_i(\bm{A}) - \bm{U}_i \bm{U}_i^\T \bm{R}_i\|_\F^2 
    &\leq \sum\limits_{i=1}^{2^{\ell}} \| r_i(\bm{A}) - \bm{R}_i\|_\F^2\\
    & = \sum\limits_{i,j=1 , i \neq j}^{2^{\ell}} \| \bm{A}_{i,j} - \bm{A}_{i,j} \bm{V}_j \bm{V}_j^\T\|_\F^2\\
    & = \sum\limits_{j=1}^{2^{\ell}} \| c_j(\bm{A}) - c_j(\bm{A}) \bm{V}_j \bm{V}_j^\T\|_\F^2.
\end{align*}
Then, by a similar argument as above, $C \leq  \Gamma_{\textup{c}} (1+\Gamma_{\textup{d}}) \cdot \text{OPT}^2$.
Hence, combining \Cref{eq:r+c} and our bounds for $R$ and $C$, we obtain
\begin{align*}
    \mathbb{E}\left[\|\bm{A} - (\bm{U} \bm{X} \bm{V}^\T + \bm{D}) \|_\F^2\right] &\leq (\Gamma_{\textup{c}} + \Gamma_{\textup{r}}) (1+\Gamma_{\textup{d}}) \cdot  \text{OPT}^2.
\end{align*}
\end{proof}



\subsection{Bounding the HSS optimum at different levels}\label{sec:HSS_level_optimality}

In this section, we prove the second key stepping stone towards  \Cref{theorem:HSSgreedy}. We show that the optimal HSS approximation error of $\bm{A}$ is always larger than that of $\bm{U}^\T(\bm{A} - \bm{D}) \bm{V}$, where $\bm{U}$ and $\bm{V}$ are orthonormal block--diagonal bases, and $\bm{D}$ is an arbitrary block-diagonal matrix of the appropriate size. I.e., in the telescoping factorization, the optimal HSS approximation error for $\vec A^{(\ell)}$ is at most that of $\vec A^{(\ell+1)}$. This result will be used to argue that, as long as the approximation is quasi-optimal at each level of \Cref{alg:HSS_greedy_meta} (which it is by \Cref{thm:simple_error}), then the error of the overall approximation is at most a $O(L)$ factor times the  optimum, yielding \Cref{theorem:HSSgreedy}.
\begin{lemma}\label{lemma:HSS_optimal}
    Consider a matrix $\bm{A} \in \mathbb{R}^{2^{\ell + 1}k \times 2^{\ell+1} k}$ and factors $\bm{U}, \bm{V}, \bm{D}$ as partitioned as in \Cref{def:partitioning}. 
    Then,
    \begin{equation*}
        \rev{\min\limits_{\bm{C} \in \HSS(\ell-1,k)}} \|\bm{U}^\T (\bm{A} - \bm{D}) \bm{V} - \bm{C}\|_\F \leq \rev{\min\limits_{\bm{B} \in \HSS(\ell,k)}} \|\bm{A} - \bm{B}\|_\F.
    \end{equation*}
\end{lemma}

 To prove~\Cref{lemma:HSS_optimal}, we  first show that certain affine transformations of HSS matrices are also HSS. A similar result is found in \cite[Lemma 4.6]{kressnermasseirobol:2019} and, for completeness, we include the proof in \Cref{section:appendixA}.

\begin{proposition}\label{prop:HSS_preservation}
    Consider a matrix $\bm{B} \in \HSS(\ell,k)$ and any $\bm{R}, \bm{L} \in \mathbb{R}^{2^{\ell + 1}k \times 2^\ell k}$ and $\bm{D} \in \mathbb{R}^{2^{\ell + 1}k \times 2^{\ell + 1}k}$ so that
    \begin{align*}
        \bm{R} &= \blockdiag(\bm{R}_1,\ldots,\bm{R}_{2^{\ell}}), \quad \bm{R}_i \in \mathbb{R}^{2k \times k}, \\
        \bm{L} &= \blockdiag(\bm{L}_1,\ldots,\bm{L}_{2^{\ell}}), \quad \bm{L}_i \in \mathbb{R}^{2k \times k},\\
        \bm{D} &= \blockdiag(\bm{D}_1,\ldots,\bm{D}_{2^{\ell}}), \quad \bm{D}_i \in \mathbb{R}^{2k \times 2k}.
    \end{align*}
    Then, $\bm{R}^\T (\bm{B} - \bm{D}) \bm{L} \in \HSS(\ell-1,k)$.
\end{proposition}

Now, we are ready to apply this result to prove~\Cref{lemma:HSS_optimal}.

\begin{proof}[Proof of \Cref{lemma:HSS_optimal}]
    \rev{Let $\widetilde{\bm{B}} \in \HSS(\ell,k)$ be an optimal HSS approximation to $\bm{A}$ so that $\|\bm{A} - \widetilde{\bm{B}}\|_\F^2 = \min\limits_{\bm{B} \in \HSS(\ell,k)}\|\bm{A} - \bm{B}\|_\F$.}\footnote{\rev{Recall that the result in \Cref{appendix:existence} guarantees the existence of an optimal approximation.}} 
    Define $\widetilde{\bm{E}} = \bm{A}-\widetilde{\bm{B}}$. Then,
    \begin{equation*}
        \bm{U}^\T (\bm{A}- \bm{D})\bm{V} = \bm{U}^\T (\bm{A} - \widetilde{\bm{B}} + \widetilde{\bm{B}} - \bm D) \bm V =  \bm{U}^\T(\widetilde{\bm{B}} - \bm{D})\bm{V} + \bm{U}^\T \widetilde{\bm{E}} \bm{V}. 
    \end{equation*}
    Note that $\bm{U}^\T(\widetilde{\bm{B}} - \bm{D})\bm{V} \in \HSS(\ell-1,k)$ by \Cref{prop:HSS_preservation}. Hence, 
    \begin{align*}
         \rev{\min\limits_{\bm{C} \in \HSS(\ell-1,k)}} \|\bm{U}^\T (\bm{A} - \bm{D}) \bm{V} - \bm{C}\|_\F 
         &\leq \|\bm{U}^\T (\bm{A}- \bm{D})\bm{V} - \bm{U}^\T(\widetilde{\bm{B}} - \bm{D})\bm{V}\|_\F 
         \\&=  \|\bm{U}^\T \widetilde{\bm{E}} \bm{V}\|_\F
         \\&\leq \|\widetilde{\bm{E}}\|_\F = \rev{\min\limits_{\bm{B} \in \HSS(\ell,k)} \|\bm{A} - \bm{B}\|_\F,}
    \end{align*}
    \rev{which yields the desired result.}
\end{proof}


\subsection{Analysis of \texorpdfstring{\Cref{alg:HSS_greedy_meta}}{Algorithm 3.1}}
With the results of \Cref{section:simpleoptimal,sec:HSS_level_optimality} in hand, we are now ready to prove \Cref{theorem:HSSgreedy}. 


\begin{proof}[Proof of \Cref{theorem:HSSgreedy}]
    Throughout the proof, we use the notation of \Cref{alg:HSS_greedy_meta}. We proceed by induction on $L$. If $L = 1$, an HSS matrix is an SSS matrix, i.e. $\HSS(1,k) = \SSS(1,k)$. Hence, for $L = 1$, the result follows from \Cref{thm:simple_error}.
    
    Now assume, for the sake of induction, that the theorem holds for any  approximation in  $\HSS(L-1, k)$ that is outputted by~\Cref{alg:HSS_greedy_meta}. That is, for any such approximation, the bound~\cref{eq:opt_bound} holds with multiplicative factor $(\Gamma_r + \Gamma_c)(1 + \Gamma_d) (L-1)$.
    Write $\bm{B}$ in its telescoping factorization \Cref{eq:telescoping_factorization} and $\bm{X} = (\bm U^{(L)})^\T ( \bm A - \rev{\bm{D}^{(L)}}) \bm V^{(L)}$. Then,
    \begin{align}
        \|\bm{A} - \bm{B}\|_\F^2
        \nonumber&= \|\bm{A} - (\bm{U}^{(L)} \bm{B}^{(L)} (\bm{V}^{(L)})^\T + \bm{D}^{(L)})\|_\F^2\\
        \nonumber&= \| \bm A - \bm D^{(L)} - \bm U^{(L)} \bm{X} (\bm V^{(L)})^\T \|_\F^2 \\
        \nonumber&\hspace{8em}+ \| \bm U^{(L)} \bm{X} (\bm V^{(L)})^\T - \bm U^{(L)}  {\bm B} ^{(L)} (\bm V^{(L)})^\T \|_\F^2 \\
        &= \|\bm{A}-\bm{D}^{(L)} - \bm{U}^{(L)} \bm{X} (\bm{V}^{(L)})^\T\|_\F^2 + \|\bm{X} - \bm{B}^{(L)}\|_\F^2.\label{eq:eq1}
    \end{align}
    The second equality follows from the the Pythagorean theorem, and the third equality is due to the orthogonal invariance of the Frobenius norm.
    
    Using \Cref{thm:simple_error}
    and the fact that $\HSS(L,k)\subseteq\SSS(L,k)$,
    \begin{align}
        \|\bm{A}-\bm{D}^{(L)} - \bm{U}^{(L)} \bm{X} (\bm{V}^{(L)})^\T\|_\F^2 
        \nonumber&\leq (1+\Gamma_{\textup{d}})(\Gamma_{\textup{c}} + \Gamma_{\textup{r}}) \rev{\min\limits_{\bm{C} \in \SSS(L,k)}}\|\bm{A} - \bm{C}\|_\F^2 
        \\&\leq (1+\Gamma_{\textup{d}})(\Gamma_{\textup{c}} + \Gamma_{\textup{r}}) \rev{\min\limits_{\bm{C} \in \HSS(L,k)}}\|\bm{A}-\bm{C}\|_\F^2.\label{eq:eq2}
    \end{align}
    
    Note that $\bm{B}^{(L)} \in \HSS(L-1,k)$ is the output of \Cref{alg:HSS_greedy_meta} when we run it on \rev{$\bm{X}$}. 
    Therefore, by our inductive assumption and \Cref{lemma:HSS_optimal} we have
    \begin{align}
        \|\bm{X} - \bm{B}^{(L)}\|_\F^2 
        &\leq (1+\Gamma_{\textup{d}})(\Gamma_{\textup{c}} + \Gamma_{\textup{r}})(L-1)\rev{\min\limits_{\bm{C} \in \HSS(L-1,k)}} \|\bm{X} - \bm{C}\|_\F^2\nonumber 
        \\&\leq (1+\Gamma_{\textup{d}})(\Gamma_{\textup{c}} + \Gamma_{\textup{r}})(L-1) \rev{\min\limits_{\bm{C} \in \HSS(L,k)}} \|\bm{A} - \bm{C}\|_\F^2.\label{eq:eq3}
    \end{align}
    Combining \Cref{eq:eq1,eq:eq2,eq:eq3} yields the desired result. 
\end{proof}
\begin{remark}
    \rev{The proof of \Cref{theorem:HSSgreedy} can be modified to obtain an absolute error bound, similar to those found in the literature \cite{kressnermasseirobol:2019}. Suppose that we have explicit access to the entries of the matrix $\bm{A}$ so that we can set $\Gamma_{\textup{d}} = 0$. Suppose further that we can compute orthonormal bases $\bm{U}_i^{(\ell)}$ and $\bm{V}_i^{(\ell)}$ so that
    \begin{align*}
        \|r_i(\bm{A}^{(\ell + 1)}) - \bm{U}_i^{(\ell)} \bm{U}_i^{(\ell)\T} r_i(\bm{A}^{(\ell + 1)})\|_\F &\leq \varepsilon,\\
        \|c_i(\bm{A}^{(\ell + 1)}) -  c_i(\bm{A}^{(\ell + 1)})\bm{V}_i^{(\ell)} \bm{V}_i^{(\ell)\T}\|_\F &\leq \varepsilon,
    \end{align*}
    for each $i$ and $\ell$. By the arguments of \Cref{thm:simple_error}, we have
    \begin{align*}
        \|\bm{A}^{(\ell+1)} - \bm{U}^{(\ell)}\bm{A}^{(\ell)}\left(\bm{V}^{(\ell)}\right)^\T - \bm{D}^{(\ell)}\|_\F^2 \leq& \sum\limits_{i=1}^{2^{\ell}} \left[\|r_i(\bm{A}^{(\ell + 1)}) - \bm{U}_i^{(\ell)} \bm{U}_i^{(\ell)\T} r_i(\bm{A}^{(\ell + 1)})\|_\F^2\right.\\
        &\hspace{-3em}+ \left.\|c_i(\bm{A}^{(\ell + 1)}) -  c_i(\bm{A}^{(\ell + 1)})\bm{V}_i^{(\ell)} \bm{V}_i^{(\ell)\T}\|_\F^2\right]\\
        \leq & 2\sum\limits_{i=1}^{2^{\ell}} \varepsilon^2
        = 2^{\ell+1} \varepsilon^2.
    \end{align*}
   Combining this with the  proof of \Cref{theorem:HSSgreedy},  the output of \Cref{alg:HSS_greedy_meta} satisfies
    \begin{equation*}
        \|\bm{A} - \bm{B}\|_\F \leq \bigg( 2^{L+1} + \sum\limits_{\ell=1}^{L-1} 2^{\ell + 1} \bigg)^{1/2}\cdot\varepsilon + \bigg( \sum\limits_{\ell=1}^L 2^{\ell + 1} \bigg)^{1/2}\cdot\varepsilon \leq 2^{(L+2)/2} \cdot \varepsilon = \sqrt{2N/k} \cdot \varepsilon,
    \end{equation*}
    since $N = 2^{L+1}k$. In particular, if we truncate all the singular values of each HSS-block row and HSS-block columns less than $\epsilon$, we have $\varepsilon \leq \sqrt{2k} \epsilon$ and we therefore have  $\|\bm{A} - \bm{B}\|_\F \leq 2 \sqrt{N} \epsilon$. This provides a similar bound to \cite[Theorem 4.7]{kressnermasseirobol:2019}, which gives a $O(\sqrt{N} \epsilon)$ bound in the spectral norm.}
\end{remark}

\subsection{Lower bound for \texorpdfstring{\Cref{alg:HSS_greedy_meta}}{Algorithm 3.1}}\label{section:lowerbound}

\Cref{theorem:HSSgreedy} guarantees that \Cref{alg:HSS_greedy_meta} produces an HSS approximation with error at most $O(\log(N/k))$ times the best possible error. 
One might hope to obtain an approximation with error that is arbitrarily close to optimal.
Unfortunately, this is not possible for \Cref{alg:HSS_greedy_meta}.

\begin{theorem}\label{theorem:lower_bound}
For any $\varepsilon > 0$, there exists a matrix $\bm{A}$ so that the output $\bm{B}$ of \Cref{alg:HSS_greedy_meta} for explicit matrices as outlined in \Cref{section:nearoptimal} satisfies
    \begin{equation*}
        \|\bm{A} - \bm{B}\|_\F^2  \geq (2-\varepsilon) \rev{\min\limits_{\bm{C} \in \HSS(L,k)}}\|\bm{A}-\bm{C}\|_\F^2.
    \end{equation*}
\end{theorem}

The proof, which is given in \Cref{sec:lower_bound_proof}, is fully constructive and in fact provides a \rev{lower bound} for the greedy SSS approximation algorithm used at each iteration of \cref{alg:HSS_greedy_meta}.
In general, we are unaware of any polynomial-time algorithm which can compute a $(1+\varepsilon)$ factor HSS or SSS approximation. In fact, it is even an open question whether  an any $O(1)$ approximation factor is possible for HSS matrices.

\rev{We note that our lower bound in~\cref{theorem:lower_bound} does not depend on $L$, while the error in our upper bound scales as $O(L)$. Addressing this gap and finding the correct dependence on $L$ (if there is any at all) is another interesting open problem.}

\section{A matrix-vector query algorithm}\label{section:matvec}

In \Cref{section:nearoptimal}, we analyzed the generic \Cref{alg:HSS_greedy_meta} for computing quasi-optimal HSS approximations.
We now shift focus to the case when  $\bm{A}$ can only be accessed through matvecs. In particular, we must now build near-optimal subspaces for the block-rows and block-columns at each iteration of \Cref{alg:HSS_greedy_meta} using only matrix-vector products  $\bm{x} \mapsto \bm{A} \bm{x}$ and $\bm{y} \mapsto \bm{A}^\T \bm{y}$.
We describe how to do this in~\Cref{alg:HSS_greedy_matvec}, which is a variant of the algorithm in \cite{LevittMartinsson:2024}. 

A key difference between our algorithm and  that of~\cite{LevittMartinsson:2024} is that we use new random sketches at each level, rather than reuse sketches from the first level. This increases the query complexity of our algorithm to $O(k\log(N/k))$ matrix-vector products, compared to $O(k)$ matrix-vector products in \cite{LevittMartinsson:2024}, but enables us to derive a near-optimality guarantee. Formally we show:
\begin{theorem}\label{theorem:HSS_matvec}
    If $s \geq 3k + 2$, then using $O(sL)$ matvecs and $O(NksL + Ns^2)$ additional runtime, \Cref{alg:HSS_greedy_matvec} returns an approximation $\bm{B} \in \HSS(L,k)$ to $\bm{A}$ so that \Cref{theorem:HSSgreedy} holds with
    \begin{align*}
        \Gamma_{\textup{r}} = \Gamma_{\textup{c}} &= \left(1 + \frac{2e(s-2k)}{\sqrt{(s-3k)^2-1}}\right)^2 = O\left(\left(\frac{s-2k}{s-3k}\right)^2\right), \quad 
        \Gamma_{\textup{d}} &= \frac{2k}{s-2k -1} .
    \end{align*}
    Hence, if we choose $s = 5k$, for example, $\Gamma_{\textup{r}}$, $\Gamma_{\textup{c}}$, and $\Gamma_{\textup{d}}$ are all fixed constants, so we find that, with $O(kL)$ matrix-vector products with $\bm{A}$ and $\bm{A}^\T$ \Cref{alg:HSS_greedy_matvec} returns an approximation to $\bm{A}$ so that
    \begin{equation*}
        \mathbb{E}\|\bm{A} - \bm{B}\|_\F^2 \leq \Gamma \cdot \rev{\min\limits_{\bm{C} \in \HSS(L,k)}}\|\bm{A}-\bm{C}\|_\F^2, \quad \text{where} \quad \Gamma = O(L).
    \end{equation*}
\end{theorem}
\rev{The} remainder of this section is structured as follows. 
In~\Cref{section:blocknullification}, we review the technique of block nullification, which allows us to efficiently simulate matvecs with the block-rows and block-columns of $\vec A$ using matrix-vector products with the full matrix.
In~\Cref{section:pcps}, we establish explicit error bounds (\Cref{thm:range_finder}) for the PCPS-based low-rank approximation technique that is used in \cref{alg:HSS_greedy_matvec} to obtain subspaces for the HSS block-rows and block-columns.
We then analyze an approach for approximating the diagonal blocks at each level in~\Cref{section:diagonal}. 
Finally, in~\Cref{sec:our_alg}, we prove our main result, \Cref{theorem:HSS_matvec}, by showing how the error bound for these primitives allows us to instantiate \cref{theorem:HSSgreedy}.

\subsection{Block nullification}\label{section:blocknullification}

There are many well-known algorithms for computing a low-rank approximation of a matrix from matrix-vector products; see e.g. \cite{HalkoMartinssonTropp:2011,TroppWebber:2023,Nakatsukasa:2020,TroppYurtzeverUdellCevher:2017}. In our setting, we want to compute matrix-vector products with the block-rows $r_i(\bm{A})$ and block-columns $c_i(\bm{A})$, though  we only have access to matrix-vector products with $\bm{A}$, not $r_i(\bm{A})$ and $c_i(\bm{A})$. 
A key idea in \cite{LevittMartinsson:2024} is \emph{block nullification}, a technique that implicitly performs matrix-vector products $\bm{x} \mapsto r_i(\bm{A}) \bm{x}$ and $\bm{y} \mapsto c_i(\bm{A})^\T \bm{y}$ using matrix vector products with $\bm{A}$, where $\bm{x}$ and $\bm{y}$ are standard Gaussian random vectors. To do so, the method must `nullify' the contribution of the diagonal blocks that are excluded from $r_i(\bm{A})$ and $c_i(\bm{A})$.
This section is devoted to describing this idea. 

Consider $\bm{A}^{(\ell + 1)}$ in \Cref{alg:HSS_greedy_meta} and $\bm{\Omega}^{(\ell)}\sim\operatorname{Gaussian}(2^{\ell + 1}k,s)$, where $s > 3k$, partitioned as 
\begin{equation}\label{eq:omega_partition}
\bm{\Omega}^{(\ell )}
= \begin{bmatrix}
    \bm{\Omega}_1^{(\ell)} \\ \bm{\Omega}_2^{(\ell)} \\ \vdots \\ \bm{\Omega}_{2^{\ell}}^{(\ell)}
\end{bmatrix}
,\quad
\bm{\Omega}_i^{(\ell )} \in \mathbb{R}^{2k\times s}
\end{equation}
We have access to matrix-vector products with $\bm{A}^{(\ell + 1)}$ through matrix-vector products with $\bm{A}$ by recursion: 
\begin{equation}\label{eq:matvec_recursion}
    \bm{A}^{(\ell + 1)} \bm{\Omega}^{(\ell)} = (\bm{U}^{(\ell  +1)})^\T(\bm{A}^{(\ell + 2)} \widehat{\bm{\Omega}}^{(\ell)} - \bm{D}^{(\ell+2)} \widehat{\bm{\Omega}}^{(\ell)}) \text{ where } \widehat{\bm{\Omega}}^{(\ell)}:= \bm{V}^{(\ell + 1)} \bm{\Omega}^{(\ell)},
\end{equation}
where matrix-vector products with $\bm{A}^{(\ell+2)}$ are, once again, computed by recursion. We continue until we reach $\bm{A}^{(L+1)} = \bm{A}$, where we directly compute matvecs with $\bm{A}$. Thus, each matvec with $\bm{A}^{(\ell+1)}$ requires one matvec with $\bm{A}$. 

Since $\bm{\Omega}_i^{(\ell)}\in\mathbb{R}^{2k\times s}$, it has a nullspace of dimension at least $s-2k > k$. Therefore, we can obtain a matrix $\bm{P}_i^{(\ell)}$ whose columns form an orthonormal basis for the nullspace of $\bm{\Omega}_i^{(\ell)}$. 
Then, $\bm{\Omega}_i^{(\ell)} \bm{P}_{i}^{(\ell )} = \bm{0}$. Moreover,  for all $j\neq i$,   the unitary invariance of Gaussian matrices together with the mutual independence of $\bm{\Omega}_i$ and $\bm{\Omega}_j$ imply that $\bm{\Omega}_j^{(\ell)} \bm{P}_{i}^{(\ell)}$ are independent Gaussian matrices.
Therefore, if we define 
\begin{equation}\label{eq:Ypartition}
\bm{G}_i^{(\ell)} = \begin{bmatrix}
    \bm{\Omega}_1^{(\ell)}\bm{P}_i^{(\ell)} \\ \vdots \\ \bm{\Omega}_{i-1}^{(\ell)}\bm{P}_i^{(\ell)} \\ \bm{\Omega}_{i+1}^{(\ell)}\bm{P}_i^{(\ell)} \\ \vdots \\ \bm{\Omega}_{2^\ell}^{(\ell)}\bm{P}_i^{(\ell)}\end{bmatrix}, \quad \bm{Y}^{(\ell)} = \bm{A}^{(\ell + 1)} \bm{\Omega}^{(\ell)} = \begin{bmatrix} \bm{Y}_1^{(\ell)}\\ \bm{Y}_2^{(\ell)} \\ \vdots\\ \bm{Y}_{2^{\ell}}^{(\ell)} \end{bmatrix},
\end{equation}
then $\bm{G}_i^{(\ell)}\sim\operatorname{Gaussian}(2^{\ell + 1}k-2k,s-2k)$ and
\begin{equation}\label{eqn:block_nullification_HSS_row}
    \bm{Y}_i^{(\ell)}\bm{P}_i^{(\ell)}
= r_i(\bm{A}) \bm{G}_i^{(\ell)}.
\end{equation}
Thus, we can implicitly compute Gaussian sketches of $r_i(\bm{A}^{(\ell + 1)})$ from the sketch $\bm{A}^{(\ell + 1)}\bm{\Omega}^{(\ell)}$. By an identical argument, we can implicitly compute Gaussian sketches of $c_i(\bm{A}^{(\ell + 1)})^\T$ from the sketch $(\bm{A}^{(\ell + 1)})^\T\bm{\Psi}^{(\ell)}$, where $\bm{\Psi}\sim\operatorname{Gaussian}(2^{\ell + 1}k,s)$. We use the same recursion in~\Cref{eq:matvec_recursion} to implicitly query $(\bm A^{(\ell)})^\T$ at all levels $1 \leq \ell \leq L$.

\subsection{Low-rank approximation via projection-cost-preserving sketches}\label{section:pcps}

At each iteration of \Cref{alg:HSS_greedy_meta}, we must  approximate the top $k$-dimensional singular subspaces of each HSS-block row and column, i.e., compute orthonormal bases $\bm{U}_i^{(\ell)},\bm{V}_i^{(\ell)} \in \mathbb{R}^{2k \times k}$ so that $\bm{U}_i^{(\ell)}\bm{U}_i^{(\ell)\T}r_i(\bm{A}^{(\ell + 1)})$ and $c_i(\bm{A}^{(\ell + 1)})\bm{V}_i^{(\ell)}\bm{V}_i^{(\ell)\T}$ are near-optimal low-rank approximations to $r_i(\bm{A}^{(\ell + 1)})$ and $c_i(\bm{A}^{(\ell + 1)})$, respectively. In this section, we describe a matvec method to do this using the idea of projection-cost-preserving sketches (PCPS) \cite{CohenElderMusco:2015, musco2020projection}.

In the previous section, block nullification allowed us to sketch the matrices  $r_i(\bm{A}^{(\ell + 1)})$ and $c_i(\bm{A}^{(\ell + 1)})^\T$ via matrix-vector products with $\bm{A}$ and $\bm{A}^\T$. However, given access to matrix-vector products with $\bm{A}$ and $\bm{A}^\T$, we cannot use this technique to query  the transposes $r_i(\bm{A}^{(\ell + 1)})^\T$ and $c_i(\bm{A}^{(\ell + 1)})$. Most randomized low-rank approximation algorithms rely on sketching these transposed matrices. For example, to apply the randomized SVD \cite{HalkoMartinssonTropp:2011} to $r_i(\bm{A}^{(\ell + 1)})$, we first obtain an orthonormal basis $\bm{Q}$ for $\text{range}(r_i(\bm{A}^{(\ell + 1)}) \bm{G}_i^{(\ell)})$, where $\bm{G}_i^{(\ell)}$ has $s > k$ columns. We  then obtain an orthonormal basis $\bm{Q}_k$ with $k$ columns satisfying $\text{range}(\bm{Q}_k) \subseteq \text{range}(\bm{Q})$ by computing the first $k$ left singular vectors of $\bm{Q}\left[\bm{Q}^\T r_i(\bm{A}^{(\ell + 1)})\right]$. If $s = O(k/\varepsilon)$ one can show that $\mathbb{E}\|r_i(\bm{A}^{(\ell + 1)}) - \bm{Q}_k \bm{Q}_k^\T r_i(\bm{A}^{(\ell + 1)}) \|_\F^2 \leq (1+\varepsilon)\|r_i(\bm{A}^{(\ell + 1)}) - \llbracket r_i(\bm{A}^{(\ell + 1)})\rrbracket_k \|_\F^2$ \cite{Sarlos:2006}. The issue with this approach is that it requires access to the matrix-vector products $r_i(\bm{A}^{(\ell + 1)})^\T \bm{Q}$. It is not clear how to efficiently compute such a sketch for all $i$ using a small number of matvecs with $\vec A$. 

Thus, we instead use the projection-cost-preserving sketch (PCPS) idea of \cite{CohenElderMusco:2015, musco2020projection}, which shows that the top singular vector subspace of $r_i(\bm{A}^{(\ell + 1)}) \bm{G}_i^{(\ell)}$ is itself a near-optimal subspace for $r_i(\bm{A}^{(\ell + 1)})$ (and analogously for $c_i(\bm{A}^{(\ell + 1)}$). I.e., we can compute the near-optimal bases required by \Cref{alg:HSS_greedy_meta} using only queries to $r_i(\bm{A}^{(\ell + 1)})$ and $c_i(\bm{A}^{(\ell + 1)})^\T$, obtained using block nullification.
We next state an explicit error bound for this PCPS-based approach. This bound may be of independent interest to the numerical linear algebra community, since prior bounds do not give explicit constants. 

\begin{theorem}\label{thm:range_finder}
    Let $\bm{B}\in \mathbb{R}^{m\times p}$ and suppose $\bm{\Omega}\sim\operatorname{Gaussian}(p,q)$. 
    Let $\bm{U}$ be the top $k$ left singular vectors of $\bm{B}\bm{\Omega}$; i.e. so that $\| \bm{B} \bm{\Omega} - \bm{U}\bm{U}^\T \bm{B}\bm{\Omega} \|_\F = \| \bm{B} \bm{\Omega} - \llbracket \bm{B}\bm{\Omega} \rrbracket_k \|_\F$.
    Then, if $q \geq k+2$,
    \[
    \EE\Bigl[ \| \bm{B} - \bm{U}\bm{U}^\T \bm{B} \|_\F^2 \Big]
    \leq \left(1 + \frac{2eq}{\sqrt{(q-k)^2-1}}\right)^2 \| \bm{B} - \llbracket \bm{B} \rrbracket_k \|_\F^2.
    \]
\end{theorem}

\begin{remark}
    We suspect that the dependence on $q$ and $k$ can be improved, but this result is sufficient for our purposes. 
\end{remark}

\begin{proof}
    Define $\bm{B}_1 = \bm{U}_1 \bm{\Sigma}_1 \bm{V}_1^\T$ to be the SVD of the optimal rank-$k$ approximation and let $\bm{M}_2 = \bm{B} - \bm{B}_1 = \bm{U}_2 \bm{\Sigma}_2 \bm{V}_2^\T$ be the remainder. Define $\bm{\Omega}_1 = \bm{V}_1^\T \bm{\Omega}$ and $\bm{\Omega}_2 = \bm{V}_2^\T \bm{\Omega}$, which are independent Gaussians. Now note that\footnote{\rev{This proof technique is related to the analysis of the randomized SVD and generalized Nyström approximation in \cite[Section 3.3-3.4]{Nakatsukasa:2020}, which was inspired by an observation in \cite[Section 4]{sorensenembree:2016}; the latter attributes the observation to Ilse Ipsen.}}
\begin{align*}
    \|(\bm{I} - \bm{U} \bm{U}^\T) \bm{B}\|_\F  &\leq \|(\bm{I} - \bm{U} \bm{U}^\T) \bm{B}\bm{\Omega} \bm{\Omega}_1^\dagger \bm{V}_1^\T\|_\F + \|(\bm{I} - \bm{U} \bm{U}^\T) \bm{B}(\bm{I}-\bm{\Omega} \bm{\Omega}_1^\dagger \bm{V}_1^\T)\|_\F\\
    &\leq \|(\bm{I} - \bm{U} \bm{U}^\T) \bm{B}\bm{\Omega}\|_\F \|\bm{\Omega}_1^\dagger\|_2 + \|\bm{B}(\bm{I}-\bm{\Omega} \bm{\Omega}_1^\dagger \bm{V}_1^\T)\|_\F.\numberthis\label{eq:two_terms}
\end{align*}
First, we bound the first term in \Cref{eq:two_terms}:
\begin{equation*}
    \|(\bm{I} - \bm{U} \bm{U}^\T) \bm{B}\bm{\Omega}\|_\F \|\bm{\Omega}_1^\dagger\|_2 \leq \|\bm{B} \bm{\Omega} - \bm{B}_1 \bm{\Omega}\|_\F \|\bm{\Omega}_1^{\dagger}\|_2 = \|\bm{\Sigma}_2 \bm{\Omega}_2\|_\F \|\bm{\Omega}_1^{\dagger}\|_2.
\end{equation*}
Next, we bound the second term in \Cref{eq:two_terms}:
\begin{align*}
    \|\bm{B}(\bm{I}-\bm{\Omega} \bm{\Omega}_1^\dagger \bm{V}_1^T)\|_\F &= \rev{\|\bm{B}(\bm{I} - \bm{V}_1 \bm{V}_1^\T)(\bm{I}-\bm{\Omega} \bm{\Omega}_1^\dagger \bm{V}_1^\T)\|_\F}\\
    &= \sqrt{\|\bm{B}(\bm{I} - \bm{V}_1\bm{V}_1^\T)\|_\F^2 + \rev{\|\bm{B}(\bm{I} - \bm{V}_1\bm{V}_1^\T) \bm{\Omega} \bm{\Omega}_1^{\dagger} \bm{V}_1^\T\|_\F^2}}\\
    &= \sqrt{\|\bm{\Sigma}_2\|_\F^2 + \|\bm{\Sigma}_2 \bm{\Omega}_2 \bm{\Omega}_1^{\dagger}\|_\F^2} \\
    &\leq \|\bm{\Sigma}_2\|_\F + \|\bm{\Sigma}_2 \bm{\Omega}_2 \bm{\Omega}_1^{\dagger}\|_\F\\
    & \leq \|\bm{\Sigma}_2\|_\F + \|\bm{\Sigma}_2 \bm{\Omega}_2\|_\F \|\bm{\Omega}_1^{\dagger}\|_2.
\end{align*}
Hence, in total, we have $
    \|(\bm{I} - \bm{U} \bm{U}^\T) \bm{B}\|_\F \leq \|\bm{\Sigma}_2\|_\F + 2\|\bm{\Sigma}_2 \bm{\Omega}_2\|_\F \|\bm{\Omega}_1^{\dagger}\|_2$.
Using the inequality $\mathbb{E}[(X+Y)^2] \leq \left(\sqrt{\mathbb{E}X^2} + \sqrt{\mathbb{E}Y^2}\right)^2$, we get
\begin{equation*}
    \|(\bm{I} - \bm{U} \bm{U}^\T) \bm{B}\|_\F^2 \leq \left(\|\bm{\Sigma}_2\|_\F + 2\sqrt{\mathbb{E}\left[\|\bm{\Sigma}_2 \bm{\Omega}_2\|_\F^2 \|\bm{\Omega}_1^{\dagger}\|_2^2\right]}\right)^2.
\end{equation*}
By noting that $\bm{\Omega}_1$ and $\bm{\Omega}_2$ are independent and using $\mathbb{E}\|\bm{\Sigma}_2 \bm{\Omega}_2\|_\F^2 = q \|\bm{\Sigma}_2\|_\F^2$ \cite[Proposition 10.1]{HalkoMartinssonTropp:2011} and $\mathbb{E}\|\bm{\Omega}_1^{\dagger}\|_2^2 \leq \frac{e^2q}{(q-k)^2 - 1}$ \cite[Lemma 3.1]{Nakatsukasa:2020} we have
\begin{align*}
    \mathbb{E}\|(\bm{I} - \bm{U} \bm{U}^\T) \bm{B}\|_\F^2 &\leq \left(1 + 2\sqrt{q}\cdot \frac{e \sqrt{q}}{\sqrt{(q-k)^2-1}}\right)^2\|\bm{\Sigma}_2\|_\F^2\\
    &=\left(1 + \frac{2eq}{\sqrt{(q-k)^2-1}}\right)^2\|\bm{\Sigma}_2\|_\F^2.
\end{align*}
Noting that $\|\bm{\Sigma}_2\|_\F = \|\bm{B} - \llbracket \bm{B} \rrbracket_k\|_\F$ yields the desired result. 
\end{proof}

\subsection{Approximating the diagonal blocks}\label{section:diagonal}
In this section, we demonstrate how to approximate the diagonal blocks at each iteration of \cref{alg:HSS_greedy_meta}, while satisfying the condition in \Cref{eq:diagonal} of \Cref{theorem:HSSgreedy}. The approach follows the same idea as in \cite[Section 4]{LevittMartinsson:2024}, where the authors assumed that the matrix is exactly HSS. Here, we show that this method is robust even when the matrix is only approximately HSS.

\begin{theorem}\label{theorem:diagonal}
    Consider a matrix $\bm{B} \in \mathbb{R}^{2^{\ell + 1}k \times 2^{\ell + 1}k}$ as partitioned in \Cref{def:partitioning}. Fix some bases
    \begin{align*}
    \bm{U}_1,\ldots,\bm{U}_{2^L} \quad &\text{where } \bm{U}_i \in \mathbb{R}^{2k \times k} \text{ has orthonormal columns,}\\
     \bm{V}_1,\ldots,\bm{V}_{2^L} \quad &\text{where } \bm{V}_i \in \mathbb{R}^{2k\times k} \text{ has orthonormal columns}.
\end{align*}
Let $\bm{\Omega},\bm{\Psi} \sim \operatorname{Gaussian}(2^{\ell + 1} k,s)$, where $s \geq 2k + 2$, be independent and partition
\begin{equation*}
    \bm{\Omega} = \begin{bmatrix} \bm{\Omega}_1 \\ \bm{\Omega}_2 \\ \vdots \\ \bm{\Omega}_{2^{\ell}} \end{bmatrix}, \quad\bm{\Psi} = \begin{bmatrix} \bm{\Psi}_1 \\ \bm{\Psi}_2 \\ \vdots \\ \bm{\Psi}_{2^{\ell}} \end{bmatrix}, \quad  \bm{Y} = \bm{B} \bm{\Omega} = \begin{bmatrix} \bm{Y}_1 \\ \bm{Y}_2 \\ \vdots \\ \bm{Y}_{2^{\ell}} \end{bmatrix}, \quad \bm{Z} = \bm{B}^\T \bm{\Psi} = \begin{bmatrix} \bm{Z}_1 \\ \bm{Z}_2 \\ \vdots \\ \bm{Z}_{2^{\ell}} \end{bmatrix},
\end{equation*}
where $\bm{\Omega}_i, \bm{\Psi}_i,\bm{Y}_i, \bm{Z}_i \in \mathbb{R}^{2k \times s}$. For $i = 1,\ldots,2^{\ell}$ define
\begin{equation*}
    \bm{D}_i = (\bm{I}-\bm{U}_i \bm{U}_i^\T) \bm{Y}_i \bm{\Omega}_i^{\dagger} + \bm{U}_i \bm{U}_i((\bm{I} - \bm{V}_i \bm{V}_i^\T) \bm{Z}_i \bm{\Psi}_i^{\dagger})^\T.
\end{equation*}
Then, for each $i = 1,\ldots,2^{\ell}$ we have 
\begin{align*}
    \hspace{4em}&\hspace{-4em}
    \mathbb{E}\Big[\|\bm{B}_{ii} - \bm{D}_i - \bm{U}_i \bm{U}_i^\T(\bm{B}_{ii} - \bm{D}_i)\bm{V}_i \bm{V}_i^\T\|_\F^2 \Big]
    \\&\leq \frac{2k}{s-2k - 1} \left(\|(\bm{I} - \bm{U}_i \bm{U}_i^\T) r_i(\bm{B})\|_\F^2 + \| c_i(\bm{B})(\bm{I} - \bm{V}_i \bm{V}_i^\T)\|_\F^2 \right)
\end{align*}
\end{theorem}
\begin{proof}
    For $i = 1,\ldots,2^{\ell}$ define
    \begin{equation*}
        \bm{G}_i = \begin{bmatrix} \bm{\Omega}_1 \\ \vdots \\ \bm{\Omega}_{i-1} \\ \bm{\Omega}_{i+1}\\ \vdots \\ \bm{\Omega}_{2^{\ell}} \end{bmatrix}, \quad \bm{H}_i = \begin{bmatrix} \bm{\Psi}_1 \\ \vdots \\ \bm{\Psi}_{i-1} \\ \bm{\Psi}_{i+1}\\ \vdots \\ \bm{\Psi}_{2^{\ell}} \end{bmatrix},
    \end{equation*}
    and note that $\bm{G}_i, \bm{H}_i, \bm{\Omega}_i$, and $\bm{\Psi}_i$ are independent. Furthermore, note that $\bm{\Omega}_i \bm{\Omega}_i^{\dagger} = \bm{\Psi}_i \bm{\Psi}_i^{\dagger} = \bm{I}$ almost surely. Hence,
    \begin{align*}
        (\bm{I} - \bm{U}_i \bm{U}_i^\T) \bm{Y}_i \bm{\Omega}_i^{\dagger} &= (\bm{I} - \bm{U}_i \bm{U}_i^\T)\bm{B}_{ii} + (\bm{I} - \bm{U}_i \bm{U}_i^\T) r_i(\bm{B}) \bm{G}_i \bm{\Omega}_i^{\dagger};\\
        (\bm{I} - \bm{V}_i \bm{V}_i^\T) \bm{Z}_i \bm{\Psi}_i^{\dagger} &= (\bm{I} - \bm{V}_i \bm{V}_i^\T)\bm{B}_{ii}^\T + (\bm{I} - \bm{V}_i \bm{V}_i^\T) c_i(\bm{B})^\T \bm{H}_i \bm{\Psi}_i^{\dagger}. 
    \end{align*}
    Consequently,
    \begin{align*}
        \bm{D}_i = & (\bm{I} - \bm{U}_i \bm{U}_i^\T)\bm{B}_{ii} + (\bm{I} - \bm{U}_i \bm{U}_i^\T) r_i(\bm{B}) \bm{G}_i \bm{\Omega}_i^{\dagger} + \\
        & \hspace{2em}+ \bm{U}_i \bm{U}_i^\T \bm{B}_{ii}(\bm{I} - \bm{V}_i \bm{V}_i^\T) + \bm{U}_i \bm{U}_i^\T(\bm{\Psi}_i^{\dagger})^\T \bm{H}_i^\T c_i(\bm{B}) (\bm{I} - \bm{V}_i \bm{V}_i^\T)\\
        = & \bm{B}_{ii} - \bm{U}_i \bm{U}_i^\T \bm{B}_{ii} \bm{V}_i \bm{V}_i^\T + (\bm{I} - \bm{U}_i \bm{U}_i^\T) r_i(\bm{B}) \bm{G}_i \bm{\Omega}_i^{\dagger} 
        \\& \hspace{2em}+ \bm{U}_i \bm{U}_i^\T(\bm{\Psi}_i^{\dagger})^\T \bm{H}_i^\T c_i(\bm{B}) (\bm{I} - \bm{V}_i \bm{V}_i^\T).
    \end{align*}
    Furthermore, noting that $\bm{U}_i \bm{U}_i^\T \bm{D}_{i} \bm{V}_i \bm{V}_i^\T = \bm{0}$ \rev{and applying Pythagoras theorem yields}
    \begin{align*}
        \hspace{4em}&\hspace{-4em}
        \|\bm{B}_{ii} - \bm{D}_i - \bm{U}_i \bm{U}_i^\T(\bm{B}_{ii} - \bm{D}_i)\bm{V}_i \bm{V}_i^\T\|_\F^2 
        \\&= \|(\bm{I} - \bm{U}_i \bm{U}_i^\T) r_i(\bm{B}) \bm{G}_i \bm{\Omega}_i^{\dagger}\|_\F^2 + \|\bm{U}_i \bm{U}_i^\T(\bm{\Psi}_i^{\dagger})^\T \bm{H}_i^\T c_i(\bm{B}) (\bm{I} - \bm{V}_i \bm{V}_i^\T)\|_\F^2
        \\&\leq \|(\bm{I} - \bm{U}_i \bm{U}_i^\T) r_i(\bm{B}) \bm{G}_i \bm{\Omega}_i^{\dagger}\|_\F^2 + \|(\bm{\Psi}_i^{\dagger})^\T \bm{H}_i^\T c_i(\bm{B}) (\bm{I} - \bm{V}_i \bm{V}_i^\T)\|_\F^2.
    \end{align*}
    Taking expectations, using the independence of $\bm{G}_i, \bm{H}_i, \bm{\Omega}_i$, and $\bm{\Psi}_i$ and applying \cite[Proposition 10.1-10.2]{HalkoMartinssonTropp:2011} yields the desired result. 
\end{proof}

\subsection{A matrix-vector query algorithm for HSS approximation}
\label{sec:our_alg}

We now show how to combine the ideas of \Cref{section:blocknullification,section:pcps,section:diagonal} to give an implementation of the greedy HSS approximation algorithm (\Cref{alg:HSS_greedy_meta}) in the matvec query model. This implementation is described as \Cref{alg:HSS_greedy_matvec} and analyzed in \Cref{theorem:HSS_matvec}. Our algorithm is a variant on the algorithm of~\cite{LevittMartinsson:2024} with a key difference, discussed below.

An obstacle to proving an accuracy guarantee 
for the algorithm of~\cite{LevittMartinsson:2024} is the reuse of sketches at each level of the hierarchy. In particular, while at level $L$, one can utilize block nullification to sketch the column and row blocks with Gaussian inputs, this is not the case at subsequent levels. Instead, the implicit sketches are with the same initial Gaussian vectors multiplied by successively more approximated telescoping factors recovered in previous levels~\cite[Remark 4.3]{LevittMartinsson:2024}. 
This introduces problematic  dependencies in the sketching matrices across levels.  

Another obstacle is presented when the matrix $\bm{A}$ is only approximately HSS. In this case, the sketches with $\bm{A}^{(\ell + 1)}$ are corrupted by errors accumulated from previous levels. To see this, consider the factorization computed by \Cref{alg:HSS_greedy_meta} at level $L$:
\begin{equation*}
    \bm{A} = \bm{U}^{(L)} \bm{A}^{(L)} (\bm{V}^{(L)})^\T + \bm{D}^{(L)} + \bm{E},
\end{equation*}
where $\bm{E}$ is the error due to $\bm{A}$ not being HSS. 
In \cite{LevittMartinsson:2024}, where $\bm{E} = \bm{0}$, the authors show that it is possible to compute a sketch of $\bm{A}^{(L)}$ using the formula:
\begin{equation*}
    \bm{A}^{(L)} \underbrace{\left[(\bm{V}^{(L)})^\T\bm{\Omega}\right]}_{\text{sketch matrix}} = (\bm{U}^{(L)})^\T \left[\bm{A}\bm{\Omega}- \bm{D}^{(L)} \bm{\Omega}\right].
\end{equation*}
Hence, if $\bm{A}$ is exactly HSS we can compute a sketch of $\bm{A}^{(L)}$ from a sketch of $\bm{A}$. Consequently, the algorithm in \cite{LevittMartinsson:2024}  recovers an HSS-factorization of $\bm{A}$ by post-processing the set $\{\bm{\Omega}, \bm{\Psi}, \bm{A} \bm{\Omega}, \bm{A}^\T \bm{\Psi} \}$, and all matrix-vector products  can be done in parallel. However, when $\bm{A}$ is not exactly HSS,  $\bm{E} \neq \bm{0}$, and applying the same idea gives
\begin{equation*}
    \bm{A}^{(L)} \underbrace{\left[(\bm{V}^{(L)})^\T\bm{\Omega}\right]}_{\text{sketch matrix}} + \underbrace{(\bm{U}^{(L)})^\T\bm{E} \bm{\Omega}}_{\text{error term}} = (\bm{U}^{(L)})^\T \left[\bm{A}\bm{\Omega}- \bm{D}^{(L)} \bm{\Omega}\right].
\end{equation*}
Thus, any sketches with $\bm{A}^{(L)}$ will be corrupted by the error, which propogates through subsequent levels, and one cannot theoretically guarantee the quality of the approximations needed in~\Cref{alg:HSS_greedy_meta}. 

For the above reasons, in~\Cref{alg:HSS_greedy_matvec}, we instead use fresh random sketches at each level. As a consequence, the query complexity increases to $O(Lk)$, as compared to $O(k)$ in \cite{LevittMartinsson:2024}.  Additionally, the matrix-vector products with $\bm{A}^{(\ell+1)}$ are adaptive due to the recursion in \cref{eq:matvec_recursion}, and thus all queries cannot be performed in parallel.

\begin{algorithm}
\caption{Greedy HSS approximation from matvecs}
\label{alg:HSS_greedy_matvec}
\textbf{input:} A matrix $\bm{A} \in \mathbb{R}^{N \times N}$ implicitly given through matrix-vector products. A number of columns for the sketching matrices $s$, where $s \geq 3k+2$. \\
\textbf{output:} An approximation $\bm{B} \in \HSS(L,k)$ to $\bm{A}$ with telescoping factorization as in \Cref{eq:telescoping_factorization}.
\begin{algorithmic}[1]
    \State $\bm{A}^{(L+1)} = \bm A$
    \For{$\ell = L, L-1, \ldots, 1$}
    \parState{Sample $\bm \Omega^{(\ell)}, \widetilde{\bm \Omega}^{(\ell)}, \bm  \Psi^{(\ell)}, \widetilde{\bm{\Psi}}^{(\ell)} \sim \Gaussian\left(2^{(\ell + 1)}k, s\right)$ partitioned as per \Cref{eq:omega_partition}.}
    \State Compute the following by recursion as in \Cref{eq:matvec_recursion}. Partition them as per \Cref{eq:Ypartition}: \begin{align*}&\bm Y^{(\ell)} = \bm{A}^{(\ell+1)} \bm \Omega ^{(\ell)} \qquad \qquad &\widetilde{\bm{Y}}^{(\ell)} = \bm{A}^{(\ell+1)} \widetilde{\bm{\Omega}} ^{(\ell)} \\ &\bm Z^{(\ell)} = \bm{A}^{(\ell+1)\T} \bm \Psi ^{(\ell)} \qquad  &\widetilde{\bm{Z}}^{(\ell)} = \bm{A}^{(\ell+1)\T} \widetilde{\bm{\Psi}} ^{(\ell)}\end{align*}
    \For{$i = 1,2,\ldots,2^\ell$}
    \State Let $\bm{P}_i^{(\ell)}$ be an orthonormal basis of $\text{nullspace}\left(\bm{\Omega}_i^{(\ell)}\right)$.
    \State Let $\bm U_i^{(\ell)} \in \mathbb{R}^{2k \times k}$ be the top $k$ left singular vectors of $\bm{Y}_i^{(\ell)} \bm{P}_i^{(\ell)}$.
    \State Let $\bm{Q}_i^{(\ell)}$ be an orthonormal basis of $\text{nullspace}\left(\bm{\Psi}_i^{(\ell)}\right)$.
    \State Let $\bm{V}_i^{(\ell)}\in \mathbb{R}^{2k \times k}$ be the top $k$ left singular vectors of $\bm{Z}_i^{(\ell)} \bm{Q}_i^{(\ell)}$.
    \State Compute \begin{align*}\bm D_i^{(\ell)} = & \left(\bm I - \bm U_i^{(\ell)} (\bm U_i^{(\ell)})^\T \right)\widetilde{\bm{Y}}_i^{(\ell)} (\widetilde{\bm{\Omega}} _i^{(\ell)})^\dagger\\ &+ \bm U_i^{(\ell)} (\bm U_i^{(\ell)})^\T \left[\left(\bm I - \bm V_i^{(\ell)} (\bm V_i^{(\ell)})^\T\right) \widetilde{\bm{Z}}_i^{(\ell)}(\widetilde{\bm{\Psi}} _i^{(\ell)})^\dagger\right]^\T \end{align*}

   \EndFor

   \State $\bm U^{(\ell)} = \text{blockdiag} (\bm{U}_1^{(\ell)},\ldots,\bm{U}_{2^{\ell}}^{(\ell)})$.
   \State $\bm V^{(\ell)} = \text{blockdiag} (\bm{V}_1^{(\ell)},\ldots,\bm{V}_{2^{\ell}}^{(\ell)})$.
   \State $\bm D^{(\ell)} = \text{blockdiag} (\bm{D}_1^{(\ell)},\ldots,\bm{D}_{2^{\ell}}^{(\ell)})$.
   \State $\bm A^{(\ell)} = (\bm{U}^{(\ell)})^\T \left(\bm{A}^{(\ell+1)} - \bm{D}^{(\ell)}\right)\bm{V}^{(\ell)}$. 
   \EndFor
\State $\bm{D}^{(0)} = \bm{A}^{(1)}$ \Comment{Computed by $\bm{A}^{(1)} \cdot \bm{I}_{2k}$.}
\end{algorithmic}
\end{algorithm}

\begin{proof}[Proof of \Cref{theorem:HSS_matvec}]
    First, note that by the discussion in \Cref{section:pcps}, $\bm{Y}_i^{(\ell)} \bm{P}_{i}^{(\ell)}$ is a sketch of $r_i(\bm{A}^{(\ell + 1)})$ with a standard Gaussian random matrix with $s-2k$ columns. Hence, by \Cref{thm:range_finder} with $q = s-2k$ we have
    \begin{align}\label{eq:quasi1}
        &\mathbb{E}\|r_i(\bm{A}_i^{(\ell + 1)}) - \bm{U}_i^{(\ell)}\bm{U}_i^{(\ell)\T} r_i(\bm{A}^{(\ell + 1)})\|_\F^2 
        \\\nonumber&\hspace{6em}\leq \left(1 + \frac{2e(s-2k)}{\sqrt{(s-3k)^2-1}}\right)^2\|r_i(\bm{A}_i^{(\ell + 1)})-\llbracket r_i(\bm{A}_i^{(\ell + 1)})\rrbracket_k\|_\F^2.
    \end{align}
    By a similar argument, we have
    \begin{align}\label{eq:quasi2}
        &\mathbb{E}\|c_i(\bm{A}_i^{(\ell + 1)}) - c_i(\bm{A}^{(\ell + 1)})\bm{V}_i^{(\ell)}\bm{V}_i^{(\ell)\T}\|_\F^2 
        \\\nonumber&\hspace{6em}\leq \left(1 + \frac{2e(s-2k)}{\sqrt{(s-3k)^2-1}}\right)^2 \|c_i(\bm{A}_i^{(\ell + 1)})-\llbracket c_i(\bm{A}_i^{(\ell + 1)})\rrbracket_k\|_\F^2.
    \end{align}
    In addition, by the independence of $\bm{U}_i^{(\ell)}, \bm{V}_i^{(\ell)}, \widetilde{\bm{\Omega}}^{(\ell)},$ and $\widetilde{\bm{\Psi}}^{(\ell)}$  and \Cref{theorem:diagonal},
    \begin{align}\label{eq:quasi3}
    &\mathbb{E}\left[\|\bm{B}_{ii} - \bm{D}_i - \bm{U}_i \bm{U}_i^\T(\bm{B}_{ii} - \bm{D}_i)\bm{V}_i \bm{V}_i^\T\|_\F^2|\bm{U}_i^{(\ell)}, \bm{V}_i^{(\ell)}\right] 
    \\\nonumber&\hspace{6em}\leq \frac{2k}{s-2k - 1} \left(\|(\bm{I} - \bm{U}_i \bm{U}_i^\T) r_i(\bm{B})\|_\F^2 + \| c_i(\bm{B})(\bm{I} - \bm{V}_i \bm{V}_i^\T)\|_\F^2 \right).
\end{align}
Combining \Cref{eq:quasi1,eq:quasi3} with \Cref{theorem:HSSgreedy} yields the desired result.

Note that the matrices $\bm{A}^{(\ell+1)}$ for $\ell = L,\ldots,1$ are never formed explicitly, but are only accessed  through matvecs with $\bm{A}$ as described in \cref{eq:matvec_recursion}. Hence, the cost of computing matvecs with $\bm{A}^{(\ell+1)}$ comes from performing the operations in the recursion \cref{eq:matvec_recursion} and the matvecs with $\bm{A}$. Let $c_{\ell+1}$ be the cost of computing $s$ matvecs with $\bm{A}^{(\ell+1)}$ and its transpose using the recursion in \eqref{eq:matvec_recursion}, and let $C$ be the cost of computing $s$ matvecs with $\bm{A}$ and its transpose. Then, we have $c_{\ell+1} = O(2^{\ell}k^2 s) + c_{\ell+2}$ with $c_{L+1}=C$. Consequently, summing over $\ell = 1,\ldots,L$ and using $L = \log_2(N/k)$ the total cost of computing these products is $O(Nks L + LC)$. 
For a fixed level $\ell$, the additional runtime required by the inner for-loop and sampling the random Gaussian matrices is $O(2^{\ell}(k^2s + ks^2))$. Summing over $\ell$ and using $L = \log_2(N/k)$ yields an additional runtime of $O(Nks + Ns^2) = O(Ns^2)$, assuming $s \geq k$. Hence, the total cost of \Cref{alg:HSS_greedy_matvec} is $O(sL)$ matvecs with $\bm{A}$ and its transpose, and $O(NksL + Ns^2)$ additional runtime. 
\end{proof}

\section{Numerical experiments}\label{sec:experiments}
In this section, we compare our \Cref{alg:HSS_greedy_matvec} to \cite[Algorithm 3.1]{LevittMartinsson:2024}, on which it is based. 
All experiments were performed on a Macbook, and Python scripts to reproduce the figures in this section can be found at \url{https://github.com/NoahAmsel/HSS-approximation}.

Recall that \cite[Algorithm 3.1]{LevittMartinsson:2024} reuses the same random sketching matrix across levels. Therefore, it only requires only $O(k)$ matvec queries with $\bm{A}$, which compares favorably to $O(Lk)$ matvec queries required by \Cref{alg:HSS_greedy_matvec}. 
However, as discussed in \Cref{sec:our_alg}, the reuse of randomness leads to difficult to analyze dependencies in error terms across levels. As such, no strong theoretical approximation guarantees are known for \cite[Algorithm 3.1]{LevittMartinsson:2024} and, as we will see shortly, for a fixed sketch size, the algorithm is indeed less accurate than our variant that uses fresh sketches for each level. It thus remains an interesting open question to develop an $O(k)$  query algorithm that matches the theoretical performance of \Cref{alg:HSS_greedy_matvec}.

 
To investigate the effect reusing random sketches across levels, we run both methods on a variety of different test matrices, discussed below.

\subsubsection*{Inverse of a banded matrix}
A matrix $\bm{M}$ is said to have bandwidth $b$ if $\bm{M}_{ij} = 0$ whenever $|i-j|  > b$. 
Let $N = 2^{L+1}$ and let $\bm{M} \in \R^{N \times N}$ have bandwidth $b = 2k + 1$. Then $\bm{A} := \bm{M}^{-1}$ is HSS with rank $2k$ \cite[Section 3.9]{Hackbusch:2015}. We generate random symmetric banded $\bm{M}$ and approximate $\bm{A}$ with a matrix $\bm{B}\in \HSS(L,k)$ (i.e., smaller rank than needed to achieve zero error). Matrix-vector products with $\bm{A}$ are performed using a sparse Cholesky factorization of $\bm{M}$. In our experiments, we set $b = 17, N = 4096, L = 9$, and $k = 8$. 

\subsection*{Schur complement of grid graph Laplacian}
Our second model problem comes from \cite{LevittMartinsson:2024} and is inspired by sparse direct solvers based on nested dissection.
Consider an $N \times 51$ grid graph $G$ whose vertices are partitioned into sets $V_1, V_2$ and $V_3$ as shown in \Cref{fig:enter-label}.
$G$'s Laplacian matrix can be partitioned into blocks as follows:
\[ \bm{L}_G = \begin{bmatrix}
\bm{L}_{11} & \bm{0} & \bm{L}_{13} \\
\bm{0} & \bm{L}_{22} & \bm{L}_{23} \\
\bm{L}_{31} & \bm{L}_{32} & \bm{L}_{33} \\
\end{bmatrix} \]
The zero blocks appear because $V_3$ is a graph separator.
Due to this structure, the Schur complement of $\bm{L}_G$ with respect to the top-left block, $\blockdiag(\bm{L}_{11},\bm{L}_{22})$, is the $N\times N$ matrix $\bm{A} = \bm{L}_{33} - \bm{L}_{31}\bm{L}_{11}^{-1}\bm{L}_{13} - \bm{L}_{32}\bm{L}_{22}^{-1}\bm{L}_{23}$.
Efficient products with $\bm{A}$ are available from sparse Cholesky factorizations of $\bm{L}_{11}$ and $\bm{L}_{22}$. In our experiments, we set $N = 1280$ and  approximate $\vec{A}$ with $\bm{B}\in \HSS(L,k)$ where $L = 9$, and $k = 8$. 

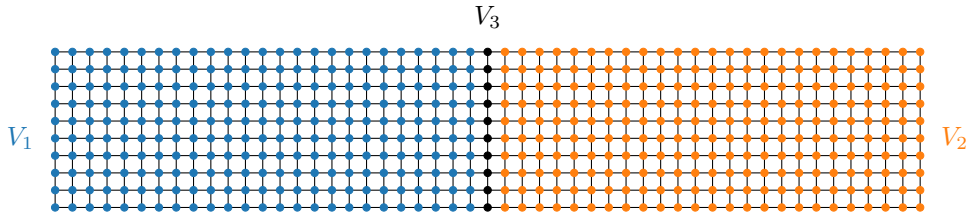
\begin{figure}
    \centering
    \begin{tikzpicture}[scale=.23]
    \def\height{10}
    \def\width{51}
    \def\splitl{25}
    \def\split{26}
    \def\splitr{27}

    \foreach \x in {1,...,\width}{
        \draw (\x,1) -- (\x,\height);
    }

    \foreach \y in {1,...,\height}{
        \draw (1,\y) -- (\width,\y);
    }

    \foreach \x in {1,...,\splitl}{
        \foreach \y in {1,...,\height}{
            \filldraw[mplb] ({\x},\y) circle (6pt);
        }
    }
    \foreach \y in {1,...,\height}{
        \filldraw[black] (\split,\y) circle (6pt);
    }
    \foreach \x in {\splitr,...,\width}{
        \foreach \y in {1,...,\height}{
            \filldraw[mplo] (\x,\y) circle (6pt);
        }
    }

    \node[text=mplb] at (-1,5) {$V_1$};
    \node[text=mplo] at (53,5) {$V_2$};
    \node[text=black] at (26,12) {$V_3$};
\end{tikzpicture}
    \caption{$N \times 51$ grid graph used for the second model problem, following the numerical experiments in \cite{LevittMartinsson:2024}. $N=10$ is pictured.}
    \label{fig:enter-label}
\end{figure}

\subsection*{Boundary integral equation}
Let $D$ be a domain in the plane with boundary $\partial D$ and let $f : \partial D \to \R$.
Parameterizing the solution using a single layer potential, it can be shown that the Laplace problem on $D$ with Neumann boundary condition $f$ is equivalent to the following integral equation for $u : \partial D \to \R$:
\begin{equation}\label{eq:bie}
    \frac12 \sigma(x) - \frac1{2\pi}\int_{\partial D} \frac{n(x) \cdot (x-y)}{\|x-y\|^2}\sigma(y) dy = f(x) \qquad \text{for all } x \in \partial D.
    \end{equation}
Here, $n(x)$ is the outwards pointing unit normal vector to $\partial D$ at $x$.
When $\partial D$ is discretized using $N$ quadrature nodes via the Nyström method, this becomes a linear system in $N$ variables.
Using the  chunkIE library \cite{chunkie}, we construct the matrix $\bm{A}$ for this linear system where $D$ is the star-shaped domain in \Cref{fig:star_domain}. In our experiments, we set $N = 1664$ and  approximate $\vec{A}$ with $\bm{B}\in \HSS(L,k)$ where $L = 5,$ and $k = 30$. 


\begin{figure}
\centering
\includegraphics[width=0.35\textwidth]{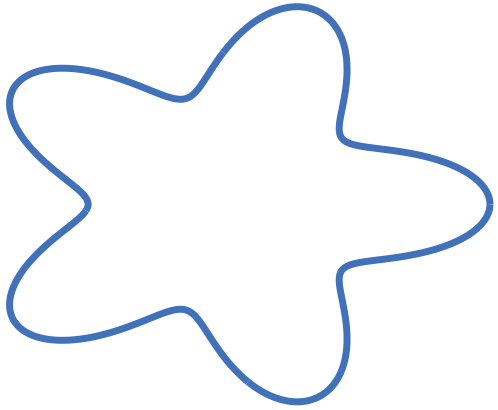}
\caption{Domain $D$ used for the boundary integral equation model problem in \Cref{eq:bie}.}
\label{fig:star_domain}
\end{figure}

\subsection*{Hard matrix from \Cref{theorem:lower_bound}}
This is the matrix defined in \Cref{eq:hard_instance}, which is far from HSS and is designed to be difficult, even for the explicit-access version of our algorithm. We set $N=32, L=4, k=1,$ and $\delta=0.1$, where $\delta$ is the parameter defined in the proof of \Cref{theorem:lower_bound}; see \Cref{section:lowerbound} and \Cref{sec:lower_bound_proof} for details. 
\begin{figure}
\begin{center}
\begin{subfigure}[b]{\textwidth}
\centering
\includegraphics[width=\textwidth]{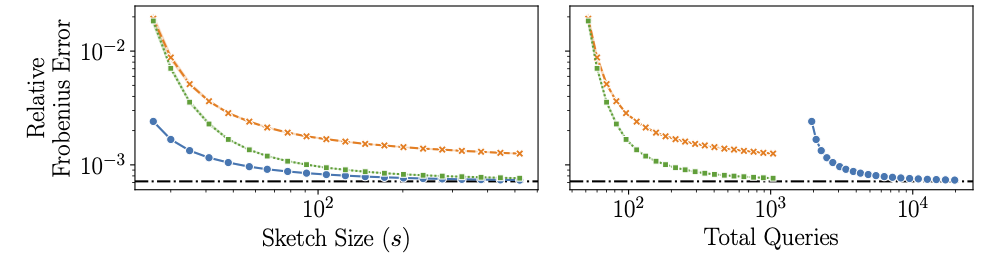}
\caption{Inverse of a banded matrix}\label{fig:inversebanded}
\end{subfigure}
\begin{subfigure}[b]{\textwidth}
\centering
\includegraphics[width=\textwidth]{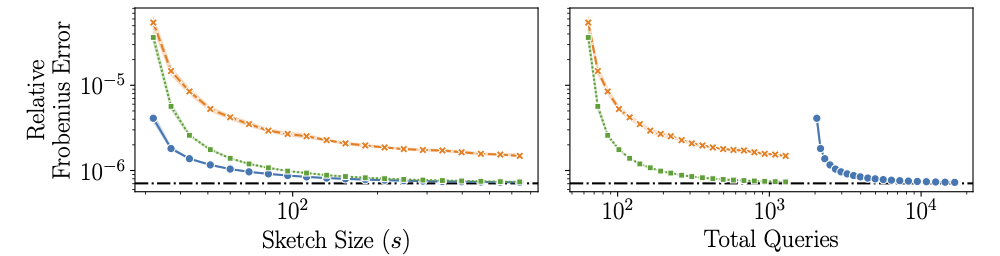}
\caption{Schur complement of graph Laplacian.}\label{fig:schurcomplement}
\end{subfigure}
\begin{subfigure}[b]{\textwidth}
\centering
\includegraphics[width=\textwidth]{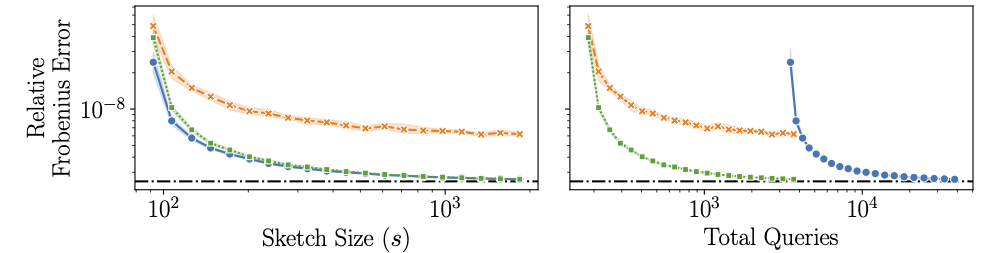}
\caption{Boundary integral equation.}\label{fig:bie}
\end{subfigure}
\begin{subfigure}[b]{\textwidth}
\centering
\includegraphics[width=\textwidth]{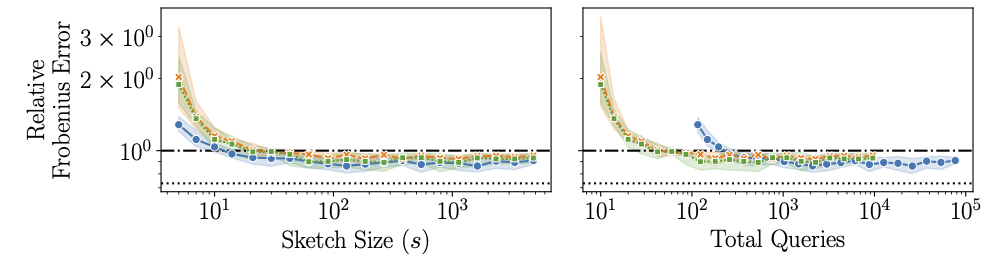}
\caption{Hard construction from \Cref{theorem:lower_bound} as defined in \Cref{eq:hard_instance}. The relative error of the optimal HSS approximation is known to be $\approx 1/\sqrt{2}$ for small $\delta$, plotted as the dotted line \includegraphics{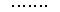}.}
\label{fig:experiment hard}
\end{subfigure}
\end{center}
\caption{Experimental results of HSS approximation on four model problems. Each line shows a different algorithm: \includegraphics{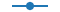}\Cref{alg:HSS_greedy_matvec} (fresh sketches at each level\rev{; requires $4sL$ matvec queries}), \includegraphics{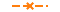} \cite[Algorithm 3.1]{LevittMartinsson:2024} (reused sketches across levels; basis found by pivoted QR\rev{; requires $4s$ matvec queries}), \includegraphics{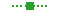}reused sketches with the basis found by SVD \rev{(requires $4s$ matvec queries)}, \includegraphics{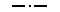}the version of \Cref{alg:HSS_greedy_meta} for explicit matrices (page \pageref{page:explicit}). Shaded regions show range over ten runs, though it is sometimes imperceptible.}
\label{fig:all_numerical}
\end{figure}

\subsection{Results}
We implement and run both \cite[Algorithm 3.1]{LevittMartinsson:2024} and \Cref{alg:HSS_greedy_matvec} on the test problems described above.
To isolate the effect of the reuse of random sketches, we also address another implementation difference between the algorithms. Specifically, \cite{LevittMartinsson:2024} computes the bases $\bm{U}^{(\ell)}$ and $\bm{V}^{(\ell)}$ using a pivoted QR decomposition, while \Cref{alg:HSS_greedy_matvec} uses the SVD. To remove any influence this minor difference in implementation may have on the results, we also implement a modified version of \cite[Algorithm 3.1]{LevittMartinsson:2024} in which we replace the pivoted QR step with an SVD. 
We run \Cref{alg:HSS_greedy_matvec} (fresh sketches), \cite[Algorithm 3.1]{LevittMartinsson:2024} (reused sketches, QR), and our modification of the latter (reused sketches, SVD) with a range of sketch sizes $s$. Following \Cref{theorem:HSS_matvec}, we  set the sketch size $s$ to be at least $3k + 2$.
For each $s$, we run each algorithm ten times and measure the relative errors $\|\bm{A} - \bm{B}\|_\F/\|\bm{A}\|_\F$. 

As a baseline, we also run the version of \Cref{alg:HSS_greedy_meta} for explicit matrices as outlined in \Cref{section:nearoptimal}. This version computes rank-$k$ truncated SVDs of the block-rows and block-columns of $\bm{A}^{(\ell)}$, and therefore requires explicit access to the entries of $\bm{A}$. While it does not necessary return an \emph{optimal} HSS approximation (we do not know any efficient algorithms to do so), we expect it to be more accurate than the matvec methods that rely on random sketching in most cases, so provides a target accuracy for those methods.

\Cref{fig:all_numerical} plots the results of all experiments.
For each test matrix, the left side plots the relative approximation error of each algorithm as a function of the sketch size $s$. The right side plots it as a function of the total number of matrix-vector queries with $\bm{A}$ and $\bm{A}^\T$ across all $L$ levels of the algorithm (``Total Queries''), which is $4s$ for \cite[Algorithm 3.1]{LevittMartinsson:2024} and $4Ls$ for \Cref{alg:HSS_greedy_matvec}.
Lines show the mean over ten runs of each algorithm, and shaded regions show the range.

The results show that the reuse of randomness in \cite[Algorithm 3.1]{LevittMartinsson:2024} does have an adverse effect on the error as compared to \Cref{alg:HSS_greedy_matvec}. 
However, for a fixed number of \emph{total} matrix-vector products, \cite[Algorithm 3.1]{LevittMartinsson:2024} has lower error than \Cref{alg:HSS_greedy_matvec}. Hence, while necessary for our theoretical guarantees, for the problems tested, \rev{using independent} sketches per level does not lead to a more efficient algorithm. \rev{In practice, we thus recommend reusing sketches across levels.} This finding leaves us \rev{hopeful} that it might be possible to improve on our matvec bound of $\rev{O(k \log(N/k))}$ by the logarithmic factor via a more refined analysis of sketch reuse. \rev{Furthermore, the experiments indicate that the modification of \cite[Algorithm 3.1]{LevittMartinsson:2024}, which computes the bases using the SVD, performs better than the version using the pivoted QR algorithm.}

In our analysis, we view \Cref{alg:HSS_greedy_matvec} as a randomized approximation to the explicit-access version of the method in \Cref{alg:HSS_greedy_meta}.
Interestingly, \Cref{fig:experiment hard} shows that in some cases, the matrix-vector query version can perform slightly better than the explicit one.
For the matrix \Cref{eq:hard_instance}, the randomness of the sketched version seems to help the algorithm avoid getting stuck in a sort of ``local minimum''. Specifically, the randomixed algorithm will not find the exact top singular vectors when constructing $\bm{U}^{(L)}$ and $\bm{V}^{(L)}$ in the first level, which by construction of this matrix, would yield a suboptimal approximation.
In effect, the randomness of \Cref{alg:HSS_greedy_matvec} washes out the contribution of $\delta$ in \Cref{eq:hard_instance}, whereas the explicit version is tricked. Due to the adversarial nature of this construction, we believe the randomized matvec version is unlikely to ever perform better in practice, but it is interesting to observe nonetheless.

\begin{paragraph}{\rev{Acknowledgments}} 
\rev{We thank the referees for helpful comments on this work. We are indebted to an anonymous referee for outlining how to prove the existence of an optimal HSS approximation.}
\end{paragraph}
\bibliographystyle{siam}
\bibliography{refs}

\vfill
\section*{Disclaimer}

This paper was prepared for informational purposes by the Global Technology Applied Research center of JPMorgan Chase \& Co. This paper is not a merchandisable/sellable product of the Research Department of JPMorgan Chase \& Co. or its affiliates. Neither JPMorgan Chase \& Co. nor any of its affiliates makes any explicit or implied representation or warranty and none of them accept any liability in connection with this paper, including, without limitation, with respect to the completeness, accuracy, or reliability of the information contained herein and the potential legal, compliance, tax, or accounting effects thereof. This document is not intended as investment research or investment advice, or as a recommendation, offer, or solicitation for the purchase or sale of any security, financial instrument, financial product or service, or to be used in any way for evaluating the merits of participating in any transaction.
\clearpage

\appendix

\section{Additional proofs}
\subsection{Proof of \texorpdfstring{\Cref{prop:HSS_preservation}}{Lemma 3.5}}\label{section:appendixA}
In this section, we provide the proof \Cref{prop:HSS_preservation}. We begin with \Cref{lemma:HSS_preservation_step1} as a stepping stone. In essence, this results show that the basis computed in \Cref{alg:HSS_greedy_meta} do not need to be orthonormal, and can always be made an orthonormalization process, and that diagonal blocks can be added while preserving the telescoping factorization. Such comments have been made before, see e.g. \cite[Remark 3.3]{CasulliKressnerRobol:2024} and \cite[Section 4.2]{Martinsson:2011}.

\begin{lemma}\label{lemma:HSS_preservation_step1}
    Consider a matrix $\bm{B} \in \HSS(\ell,k)$ and any $\bm{R}, \bm{L} \in \mathbb{R}^{2^{\ell+1}k \times 2^{\ell+1}k}$ so that
    \begin{align*}
        \bm{R} &= \blockdiag(\bm{R}_1,\ldots,\bm{R}_{2^{\ell}}) \quad \text{where } \bm{R}_i \in \mathbb{R}^{2k \times 2k}, \\
        \bm{L} &= \blockdiag(\bm{L}_1,\ldots,\bm{L}_{2^{\ell}}) \quad \text{where } \bm{L}_i \in \mathbb{R}^{2k \times 2k},\\
        \bm{D} &= \blockdiag(\bm{D}_1,\ldots,\bm{D}_{2^{\ell}}) \quad \text{where } \bm{D}_i \in \mathbb{R}^{2k \times 2k}
    \end{align*}
    Then, $\bm{R}^\T \bm{B} \bm{L} + \bm{D} \in \HSS(\ell,k)$.
\end{lemma}

\begin{proof}
The proof will be by induction on $\ell$. If $\ell = 1$ then the factorization is equivalent a HODLR matrix
\begin{equation}\label{eq:L=1}
    \bm{B} = \begin{bmatrix} \bm{B}_{1,1} & \bm{B}_{1,2}\\ \bm{B}_{2,1} & \bm{B}_{2,2} \end{bmatrix},
\end{equation}
where $\bm{B}_{i,j} \in \mathbb{R}^{2k \times 2k}$ and $\rank(\bm{B}_{i,j}) \leq k$ if $i \neq j$. Then,
\begin{equation*}
    \bm{R}^\T \bm{B} \bm{L} + \bm{D} = \begin{bmatrix} \bm{R}_1^\T\bm{B}_{1,1} \bm{L}_1 + \bm{D}_1 & \bm{R}_1^\T\bm{B}_{1,2} \bm{L}_2\\ \bm{R}_2^\T\bm{B}_{2,1}\bm{L}_1 & \bm{R}_2^\T\bm{B}_{2,2}\bm{L}_2 + \bm{D}_2 \end{bmatrix},
\end{equation*}
which is also HODLR of the same structure as \Cref{eq:L=1}. Hence, the result holds for $\ell =1$. 

Now assume that the result holds for $\ell - 1$ levels. With this assumption, we will show that the result holds for $\ell$ levels. Since $\bm{B} \in \HSS(\ell,k)$ it has a telescoping factorization
\begin{equation*}
    \bm{B} = \bm{U}^{(\ell)} \bm{B}^{(\ell)} (\bm{V}^{(\ell)})^\T + \bm{D}^{(\ell)}
\end{equation*}
where $\bm{B}^{(\ell)} \in \HSS(\ell-1,k)$ and the remaining factors are given by 
\begin{align*}
    \bm{U}^{(\ell)} &= \blockdiag\left(\bm{U}_1^{(\ell)},\ldots,\bm{U}_{2^{\ell}}^{(\ell)}\right), \quad \bm{U}_i^{(\ell)} \in \mathbb{R}^{2k\times k} \text{ has orthonormal columns,}\\
    \bm{V}^{(\ell)} &= \blockdiag\left(\bm{V}_1^{(\ell)},\ldots,\bm{V}_{2^{\ell}}^{(\ell)}\right), \quad \bm{V}_i^{(\ell)} \in \mathbb{R}^{2k\times k} \text{ has orthonormal columns,}\\
    \bm{D}^{(\ell)} &= \blockdiag\left(\bm{D}_1^{(\ell)},\ldots,\bm{D}_{2^{\ell}}^{(\ell)}\right), \quad  \bm{D}_i^{(\ell)} \in \mathbb{R}^{2k\times 2k}.
\end{align*}
For $i = 1,\ldots,2^\ell$ factorize $\bm{R}_i^\T \bm{U}_i^{(\ell)} = \widetilde{\bm{U}}_i^{(\ell)}\widetilde{\bm{R}}_i$ and $\bm{L}_i^\T \bm{V}_i^{(\ell)} = \widetilde{\bm{V}}_i^{(\ell)}\widetilde{\bm{L}}_i$ where $\widetilde{\bm{U}}_i^{(\ell)}, \widetilde{\bm{V}} _i^{(\ell)} \in \mathbb{R}^{2k \times k}$ have orthonormal columns and $\widetilde{\bm{R}}_i, \widetilde{\bm{L}}_i \in \mathbb{R}^{k \times k}$.  Define
\begin{align*}
    \widetilde{\bm{U}}^{(\ell)} &= \blockdiag\left(\widetilde{\bm{U}}_1^{(\ell)},\ldots,\widetilde{\bm{U}}_{2^{\ell}}^{(\ell)}\right),\\
    \widetilde{\bm{V}}^{(\ell)} &= \blockdiag\left(\widetilde{\bm{V}}_1^{(\ell)},\ldots,\widetilde{\bm{V}}_{2^{\ell}}^{(\ell)}\right),\\
    \widetilde{\bm{D}}^{(\ell)} &= \blockdiag\left(\bm{R}_1^\T\bm{D}_1^{(\ell)}\bm{L}_1 + \bm{D}_1,\ldots,\bm{R}^\T_{2^{\ell}}\bm{D}_{2^{\ell}}^{({\ell})}\bm{L}_{2^{\ell}} + \bm{D}_{2^{\ell}}\right),\\
    \widetilde{\bm{R}} &= \blockdiag\left(\begin{bmatrix} \widetilde{\bm{R}}_1 & \\ & \widetilde{\bm{R}}_2 \end{bmatrix}, \ldots,\begin{bmatrix} \widetilde{\bm{R}}_{2^{\ell}-1} & \\ & \widetilde{\bm{R}}_{2^{\ell}} \end{bmatrix}\right),\\
    \widetilde{\bm{L}} &= \blockdiag\left(\begin{bmatrix} \widetilde{\bm{L}}_1 & \\ & \widetilde{\bm{L}}_2 \end{bmatrix}, \ldots,\begin{bmatrix} \widetilde{\bm{L}}_{2^{\ell}-1} & \\ & \widetilde{\bm{L}}_{2^{\ell}} \end{bmatrix}\right),
\end{align*}
where $\widetilde{\bm{R}}, \widetilde{\bm{L}} \in \mathbb{R}^{2^{\ell}k \times 2^{\ell} k}$ contain $2^{\ell-1}$ diagonal blocks of sizes $2k \times 2k$. Consequently, 
\begin{align*}
    \bm{R}^\T \bm{B} \bm{L} + \bm{D} &= \bm{R}^\T (\bm{U}^{(\ell)} \bm{B}^{(\ell)} (\bm{V}^{(\ell)})^\T + \bm{D}^{(\ell)})\bm{L}+ \bm{D}\\
    &=  \widetilde{\bm{U}}^{(\ell)} \widetilde{\bm{R}}^\T \bm{B}^{(\ell)} \widetilde{\bm{L}} (\widetilde{\bm{V}}^{(\ell)})^\T + \bm{R}^\T \bm{D}^{(\ell)}\bm{L} + \bm D\\
    &= \widetilde{\bm{U}}^{(\ell)} \widetilde{\bm{R}}^\T \bm{B}^{(\ell)} \widetilde{\bm{L}} (\widetilde{\bm{V}}^{(\ell)})^\T + \widetilde{\bm{D}}^{(\ell)}. 
\end{align*}
where $\widetilde{\bm{D}}^{(\ell)}:= \bm{R}^\T \bm{D} ^{(\ell)}\bm{L} + \bm D$ is block-diagonal. By our inductive hypothesis $\widetilde{\bm{B}}^{(\ell)} = \widetilde{\bm{R}}^\T \bm{B}^{(\ell)} \widetilde{\bm{L}}\in \HSS(\ell-1,k)$ and therefore has a telescoping factorization. Hence, 
\begin{equation*}
     \bm{R}^\T \bm{B} \bm{L} + \bm{D} = \widetilde{\bm{U}}^{(\ell)} \widetilde{\bm{B}}^{(\ell)}(\widetilde{\bm{V}}^{(\ell)})^\T + \widetilde{\bm{D}}^{(\ell)},
\end{equation*}
gives a telescoping factorization of $\bm{R}^\T \bm{B} \bm{L} + \bm{D}$, and we therefore have $\bm{R}^\T \bm{B} \bm{L} + \bm{D} \in \HSS(\ell,k)$, as required.                                   

\end{proof}

\begin{lemma}[\Cref{prop:HSS_preservation} restated]\label{lemma:HSS_preservation_appendix}
    Consider a matrix $\bm{B} \in \HSS(\ell,k)$ and any $\bm{R}, \bm{L} \in \mathbb{R}^{2^{\ell + 1}k \times 2^\ell k}$ and $\bm{D} \in \mathbb{R}^{2^{\ell + 1}k \times 2^{\ell + 1}k}$ so that
    \begin{align*}
        \bm{R} &= \blockdiag(\bm{R}_1,\ldots,\bm{R}_{2^{\ell}}), \quad \bm{R}_i \in \mathbb{R}^{2k \times k}, \\
        \bm{L} &= \blockdiag(\bm{L}_1,\ldots,\bm{L}_{2^{\ell}}), \quad \bm{L}_i \in \mathbb{R}^{2k \times k},\\
        \bm{D} &= \blockdiag(\bm{D}_1,\ldots,\bm{D}_{2^{\ell}}), \quad \bm{D}_i \in \mathbb{R}^{2k \times 2k}.
    \end{align*}
    Then, $\bm{R}^\T (\bm{B} - \bm{D}) \bm{L} \in \HSS(\ell-1,k)$.
\end{lemma}

\begin{proof}
    We can write $\bm{B}$ using the telescoping factorization
    \begin{equation*}
    \bm{B} = \bm{U}^{(\ell)} \bm{B}^{(\ell)} (\bm{V}^{(\ell)})^\T + \bm{D}^{(\ell)}
\end{equation*}
where $\bm{B}^{(\ell)} \in \HSS(\ell-1,k)$ and the remaining factors are given by 
\begin{align*}
    \bm{U}^{(\ell)} &= \blockdiag\left(\bm{U}_1^{(\ell)},\ldots,\bm{U}_{2^{\ell}}^{(\ell)}\right) \quad \text{where } \bm{U}_i^{(\ell)} \in \mathbb{R}^{2k\times k} \text{ has orthonormal columns,}\\
    \bm{V}^{(\ell)} &= \blockdiag\left(\bm{V}_1^{(\ell)},\ldots,\bm{V}_{2^{\ell}}^{(\ell)}\right) \quad \text{where } \bm{V}_i^{(\ell)} \in \mathbb{R}^{2k\times k} \text{ has orthonormal columns,}\\
    \bm{D}^{(\ell)} &= \blockdiag\left(\bm{D}_1^{(\ell)},\ldots,\bm{D}_{2^{\ell}}^{(\ell)}\right) \quad \text{where } \bm{D}_i^{(\ell)} \in \mathbb{R}^{2k\times 2k}.
\end{align*}
Hence,
\begin{equation*}
    \bm{R}^\T (\bm{B} - \bm{D}) \bm{L} = \bm{R}^\T \bm{U}^{(\ell)} \bm{B}^{(\ell)} (\bm{V}^{(\ell)})^\T \bm{L} + \bm{R}^\T (\bm{D}^{(\ell)} - \bm{D}) \bm{L}.
\end{equation*}
Note that $\bm{R}^\T (\bm{D}^{(\ell)} - \bm{D}) \bm{L}, \bm{R}^\T \bm{U}^{(\ell)}, \bm{L}^\T \bm{V}^{(\ell)} \in \mathbb{R}^{2^{\ell} k \times 2^{\ell} k}$ are block-diagonal with $2^{\ell}$ blocks of sizes $k \times k$. Therefore, they are also block-diagonal with $2^{\ell-1}$ blocks of sizes $2k \times 2k$. Since $\bm{B}^{(\ell)} \in \HSS(\ell-1,k)$ by \Cref{lemma:HSS_preservation_step1} we have $\bm{R}^\T (\bm{B} - \bm{D}) \bm{L} \in \HSS(\ell-1,k)$, as required. 
\end{proof}

\subsection{Proof of \texorpdfstring{\Cref{theorem:lower_bound}}{Theorem 3.8}}
\label{sec:lower_bound_proof}
We will construct an example matrix for which our algorithms output a suboptimal HSS approximation, even given full access to the entries.
This matrix is special in that we are able to find the optimal HSS approximation explicitly.
This construction shows that our algorithms are incapable of producing an approximation that is arbitrarily close to optimal.

\begin{proof}
For any $L$, let $N = 2^{L+1}$.
Construct $\vec{A} \in \R^{N \times N}$ as follows.
Divide $\vec{A}$ into blocks of size $\vec{A}_{i,j} \in \R^{2 \times 2}$ for $i,j \in \{1, \ldots, 2^L \}$, and set
\begin{equation}\label{eq:hard_instance} \vec{A}_{i,j} = \begin{cases}\begin{bmatrix}0&1+\delta \\ 1 & 0\end{bmatrix} & i+j = 2^L+1 \\ \vec{I}_{2 \times 2} & \text{o.w.}\end{cases}\end{equation}
for $\delta \in (0, 1)$. That is, the diagonal blocks are all zero, the anti-diagonal blocks are perturbed exchange matrices, and the remaining blocks are the identity.
For instance, for $L=2$, $\vec{A}$ is the following $8 \times 8$ matrix:
\[ \vec{A} = 
\left[\begin{array}{c|c|c|c}
\vec{I} & \vec{I} & \vec{I} & \begin{matrix} & 1+\delta \\ 1 &\end{matrix} \\
\hline
\vec{I} & \vec{I} & \begin{matrix} & 1+\delta \\ 1 &\end{matrix} & \vec{I} \\
\hline
\vec{I} & \begin{matrix} & 1+\delta \\ 1 &\end{matrix} & \vec{I} & \vec{I} \\
\hline
\begin{matrix} & 1+\delta \\ 1 &\end{matrix} & \vec{I} & \vec{I} & \vec{I}
\end{array}\right]
\]

We now apply \Cref{alg:HSS_greedy_meta} to this matrix with rank parameter $k=1$.
Each row of $\vec{A}$ has the same number of non-zero entries as every other row. The non-zero entries are all $1$ except that the odd-numbered rows have a single entry that is $1+\delta$. Thus, the top singular vector for each block row is $\vec{U}_i^{(L)} = \begin{bmatrix}1 & 0\end{bmatrix}^\T$.
Each column of $\vec{A}$ has the the same number of non-zero entries. They are all $1$ except that the \emph{even}-numbered columns have a single $1+\delta$. Thus, the top singular vector each block column is $\vec{V}_j^{(L)} = \begin{bmatrix}0 & 1\end{bmatrix}^\T$. Most blocks in $\vec{A}$ are not in the span of these low rank factors:
\[
\left(\vec{U}_i^{(L)}\right)^\T \vec{A}_{i,j}\vec{V}_j^{(L)}
= \begin{cases}1+\delta & i+j = 2^L+1 \\ 0 & \text{o.w.} \end{cases}
\]
Regardless of what happens in the succeeding iterations of the algorithm, the output will have the form
\[ \vec{U}^{(L)} \vec{B}^{(L)} {\vec{V}^{(L)}}^\T + \vec{D}^{(L)} \]
where
\begin{align*}
    \bm{U}^{(L)} &= \blockdiag\left(\bm{U}_1^{(L)},\ldots,\bm{U}_{2^L}^{(L)}\right) \\
    \bm{V}^{(L)} &= \blockdiag\left(\bm{V}_1^{(L)},\ldots,\bm{V}_{2^L}^{(L)}\right) \\
    \bm{D}^{(L)} &= \blockdiag\left(\bm{D}_1^{(L)},\ldots,\bm{D}_{2^L}^{(L)}\right)
\end{align*}
Furthermore, the error of approximation is lower bounded by the Frobenius weight of all the blocks that are neither on the diagonal nor on the anti-diagonal. Combining, we have:
\begin{align*}
\left\|\vec{A} - \left(\vec{U}^{(L)} \vec{B}^{(L)} {\vec{V}^{(L)}}^\T + \vec{D}^{(L)}\right)\right\|_\F^2
&\geq
\sum_{i\neq j} \left\|\vec{A}_{ij} - \vec{U}_i^{(L)}\vec{B}^{(L)}_{ij}{\vec{V}_j^{(L)}}^\T\right\|_\F^2 \\
&\geq
\sum_{\substack{i\neq j \\ i+j \neq 2^L + 1}} \left\|\vec{I}_{2 \times 2} - \vec{U}_i^{(L)}\vec{B}^{(L)}_{ij}{\vec{V}_j^{(L)}}^\T\right\|_\F^2 \\
&\geq \left(2^L \cdot 2^L - 2 \cdot 2^L\right) \left\|\vec{I}_{2 \times 2}\right\|_\F^2 \\
&= 2^{2L+1} - 2^{L+2}
\end{align*}
Above, $2^L \cdot 2^L - 2 \cdot 2^L$ counts the number of blocks that are neither in the diagonal nor in the off-diagonal.

However, there exists a far better HSS approximation to $\vec{A}$. Its telescoping factorization is given as follows. For all $\ell \in \{1, \ldots, L\}$ and all $i,j \in \{1, \ldots, 2^\ell\}$,
\begin{align*}
\vec{U}_i^{(\ell)} = \vec{V}_j^{(\ell)} = \frac1{\sqrt 2}\begin{bmatrix}1 \\ 1\end{bmatrix}
\qquad \vec{D}_i^{(\ell)} = \vec{0}_{2 \times 2}
\end{align*}
and $\vec{D}^{(0)} = 2^{L-1} \cdot \vec{I}_{2 \times 2}$.
From \Cref{def:HSS}, the corresponding HSS matrix is $\vec{B}^* = \frac12 \vec{1}_{N \times N}$, where $\vec{1}_{N \times N}$ is the $N \times N$ all ones matrix.
The error of this approximation is
\begin{align*}
\|\vec{A} - \vec{B}^*\|_\F^2
&= \sum_{i,j} \left\|\vec{A}_{ij} - \frac12 \vec{1}_{2\times 2}\right\|_\F^2 \\
&= \sum_{i+j\neq 2^L-1} \left\|\vec{I}_{2\times 2} - \frac12 \vec{1}_{2\times 2}\right\|_\F^2 + \sum_{i+j=2^L-1} \left\|\begin{bmatrix}0&1+\delta \\ 1 & 0\end{bmatrix} - \frac12 \vec{1}_{2\times 2}\right\|_\F^2 \\
&= (2^L \cdot 2^L - 2^L) + 2^L (1 + \delta + \delta^2) \\
&\leq 2^{2L} + \delta \cdot 2^{L+1}
\end{align*}
Using $\delta < 1$, the approximation returned by \Cref{alg:HSS_greedy_meta} is suboptimal by a factor of at least
\[ \frac{2^{2L+1}-2^{L+2}}{2^{2L} + \delta \cdot 2^{L+1}} \geq 2 \cdot \frac{2^{L-1} - 1}{2^{L-1} + 1} > 2 - \epsilon \]
for sufficiently large $L$ ($\approx O(\log1/\epsilon)$).
This proves \Cref{theorem:lower_bound}.
\end{proof}

\section{Other hierarchical matrix classes}
\label{sec:other_hierarchical}
Due to the rich literature on hierarchical matrices, as well as applications and motivations from diverse fields, there is not always a consensus on the definitions of and distinctions between hierarchical matrix classes. 
While we define each class that we use explicitly, it is useful to understand the conceptual relation between different matrix classes.
Inspired by \cite{AshcraftButtariMary:2021}, we provide \Cref{fig:hierarchical_defs}, which gives the notation that we use and describes three key properties that characterize the relations between key matrix classes.
\begin{itemize}
    \item \textbf{Non-recursive:}
    SSS is a class is similar to HSS, but has no recursive structure and is therefore an example of a non-hierarchical (or flat) analog of HSS. 
    \item \textbf{Partition and admissibility:} 
    Only the off-diagonal blocks of HSS and SSS matrices are low-rank, so that only diagonal blocks are recursively subdivided. However, a hierarchical matrix could have a different \textit{partition}, where  a given off-diagonal block is not low-rank and is also  recursively subdivided. This arises in the context of hyperbolic PDEs~\cite{TownsendWang:2023}, or when the hierarchical matrix is generated for a higher-dimensional problem. Moreover, as defined, the HSS matrix partition is  weakly admissible, as  the low-rank blocks may contain diagonal entries. In contrast, a strongly admissible partition   only allows off-tridiagonal blocks to be low-rank. 
    \item \textbf{Shared bases:} The blocks of HSS and SSS matrices are each low-rank and the low-rank factors are shared across HSS-rows/columns. 
    In contrast, the low-rank blocks of HODLR matrices may not share any factors. 
\end{itemize}

In this paper, we focus on the development of {algorithmic and analysis techniques}, rather than generality.
As such, for clarity, we make simple assumptions in our definitions of HSS (\Cref{def:HSS}) and SSS (\Cref{def:HSS_simple}).
Our work also extends naturally to HSS matrices with differently sized blocks at a given level, and in the case that the off-diagonal blocks have different rank-parameters, it suffices to take the maximum off-diagonal block rank as the global rank parameter.
Furthermore, in \Cref{sec:BLR} we show that our analysis of SSS matrices is easily extended to $\text{BLR}^2$ (also called uniform-BLR) matrices, which are the focus of recent work \cite{PearceYesypenkoLevittMartinsson:2025}.

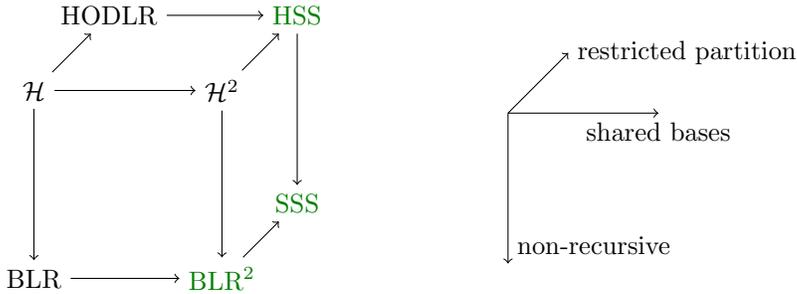
\begin{figure}
    \centering
\adjustbox{valign=c}{\begin{tikzpicture}[scale=0.50]

\node (H) at (0,5) {$\mathcal{H}$};
\node (H2) at (5,5) {$\mathcal{H}^2$};
\node (BLR) at (0,0) {$\textup{BLR}$};
\node (BLR2) at (5,0) {\hyperref[def:UBLR]{$\textup{BLR}^2$}};
\node (HODLR) at (2,7) {$\textup{HODLR}$};
\node (HSS) at (7,7) {\hyperref[def:HSS]{$\HSS$}};
\node (SSS) at (7,2) {\hyperref[def:HSS_simple]{$\SSS$}};

\draw[->] (H) -- (H2);
\draw[->] (HODLR) -- (HSS);
\draw[->] (BLR) -- (BLR2);

\draw[->] (H) -- (BLR);
\draw[->] (HSS) -- (SSS);
\draw[->] (H2) -- (BLR2);

\draw[->] (H) -- (HODLR);
\draw[->] (H2) -- (HSS);
\draw[->] (BLR2) -- (SSS);

\end{tikzpicture}}
\hspace{2cm}
\adjustbox{valign=c}{
\begin{tikzpicture}[scale=.4]
    
\draw[->] (0,5) -- (5,5) node[below] {shared bases};
\draw[->] (0,5) -- (0,0) node[above right] {non-recursive};
\draw[->] (0,5) -- (2,7) node[right] {restricted partition};

\end{tikzpicture}}
\caption{Categorization of hierarchical/rank-structured matrix classes.
Arrows indicate a matrix-class becoming more restrictive; i.e. $A\rightarrow B$ means $B\subseteq A$.
The orientation of the arrows indicates the additional property that defines the subset. 
In this paper we discuss algorithms for the $\HSS$, $\SSS$, and $\UBLR$ classes.
}
\label{fig:hierarchical_defs}
\end{figure}

\subsection{Hardness of Approximation}
Whether or not a family has the ``shared bases'' property appears to lead to the largest differences in algorithmic techniques. 
In particular, there are trivial polynomial time algorithms for computing the best HODLR, $\mathcal{H}$, and BLR matrices approximation to a given matrix.\footnote{Assuming access to a polynomial time SVD algorithm.}
In contrast, we are unaware of any (even super-polynomial time) algorithms for producing the best approximation from classes with shared bases. 
In fact, we suspect that many approximation problems involving families with shared bases are NP-hard, or hard assuming standard complexity theoretic assumptions, such as the Unique Games Conjecture.

The separation between families with and without shared bases seems to carry through to matvec query algorithms as well.
For instance, \cite{ChenDumanKelesHalikiasMuscoMuscoPersson:2025} proves that there exists an algorithm that that, given $\varepsilon > 0$, computes a HODLR approximation whose approximation error is at most $(1+\varepsilon)$ times the optimal HODLR approximation error. 
In contrast, the algorithm we describe in this paper, which we believe is the first matvec algorithm for HSS approximation with any theoretical optimality guarantees, cannot produce better than a constant-factor approximation error; see \Cref{section:lowerbound}.

\section{Bounds for uniform block low-rank approximation}
\label{sec:BLR}



In this section, we consider the class of uniform block low-rank, or $\UBLR$, matrices \cite{AshcraftButtariMary:2021,PearceYesypenkoLevittMartinsson:2025}. In particular, we will provide an error bound for \cite[Algorithm 1]{PearceYesypenkoLevittMartinsson:2025}.
Uniform block low-rank matrices matrices are slightly more general than the class $\SSS(\ell,k)$ defined in \Cref{def:HSS_simple}; a uniform BLR matrix allows the sparse remainder matrix $\bm D$ to have a different sparsity pattern than diagonal. In particular, $\bm D$ is block-diagonal in~\Cref{def:HSS_simple}, whereas, if one uses a strongly admissible partition for a uniform BLR matrix,  $\bm D$ is  block-tridiagonal.  Formally, the class of uniform block low-rank matrices are defined as follows.

\begin{definition}\label{def:UBLR}
Consider a matrix $\bm{B} \in \mathbb{R}^{bm \times bm}$.
Let $k>0$ be a rank parameter, and $S\subset \{1,\ldots, b\}\times \{1,\ldots,b\}$.
We say $\bm{B} \in \UBLR(b,m,k,S)$ if it can be written as 
\begin{equation*}
    \bm{B} = \bm{U} \bm{X} \bm{V}^\T + \bm{D},
\end{equation*}
where $\bm{X}\in\mathbb{R}^{bk\times bk}, \bm{U} \in \mathbb{R}^{bm \times bk}, \bm{V} \in \mathbb{R}^{bm \times bk},$ and $\bm{D} \in \mathbb{R}^{bm \times bm}$ have the following block-structures
\begin{align*}
    \vec{X} &= \begin{bmatrix}
        \vec{X}_{1,1} & \cdots & \vec{X}_{1,b} \\
        \vdots & \ddots &\vdots \\
        \vec{X}_{b,1} & \cdots & \vec{X}_{b,b} \\
    \end{bmatrix}
    \quad\text{where $\vec{X}_{i,j}\in\mathbb{R}^{k\times k}$,}\\
    \bm{U} &= \blockdiag\left(\bm{U}_1,\ldots,\bm{U}_{b}\right) \quad \text{where } \bm{U}_i \in \mathbb{R}^{m \times k} \text{ has orthonormal columns,}\\
     \bm{V} &= \blockdiag\left(\bm{V}_1,\ldots,\bm{V}_{b}\right) \quad \text{where } \bm{V}_i \in \mathbb{R}^{m \times k} \text{ has orthonormal columns, and}\\
     \vec{D} &= \begin{bmatrix}
        \vec{D}_{1,1} & \cdots & \vec{D}_{1,b} \\
        \vdots & \ddots &\vdots \\
        \vec{D}_{b,1} & \cdots & \vec{D}_{b,b} \\
    \end{bmatrix}
    \quad\text{where $\vec{D}_{i,j}\in\mathbb{R}^{m\times m}$ and $(i,j)\not\in S \Longrightarrow \vec{D}_{i,j} = \bm{0}$.}
\end{align*}
\end{definition}

The index subset $S$ encodes the \textit{admissibility criteria} of the partition and indicates, for example, whether or not low-rank subblocks are allowed to touch the diagonal. More specifically, $S$ consists of all the pairs of indices that correspond to inadmissible (i.e., not low-rank) blocks. 

\begin{remark}
Let $m=2k$, $b=2^\ell$, $S = \{(i,i) : i=1,\ldots, 2^\ell\}$.
    Then, from~\Cref{def:HSS_simple}, we have 
    \[
    \UBLR(b,m,k,S) = \SSS(\ell,k).
    \]
\end{remark}
As described in \cite{PearceYesypenkoLevittMartinsson:2025}, one can use the \emph{block nullification} as described in \Cref{section:blocknullification} to form random sketches of low-rank matrix blocks of $\bm{B}$ in \Cref{def:UBLR}. These sketches can then be used to compute the basis in $\bm{U}$ and $\bm{V}$ in \Cref{def:UBLR} according to the method presented in \Cref{section:pcps}. The matrix $\bm{D}$ can be computed using a similar method as described in \Cref{section:diagonal}. 

\subsection{Notation}\label{section:ublr_notation}
To facilitate our discussion we begin with introducing the necessary preliminaries and notation. For the rest of the discussion we assume that $b,m,k,$ and $S$ in \Cref{def:UBLR} are fixed. Define the following subsets of $S$:
\begin{align*}
    R_i &= \{ j \in \{1,\ldots,b\} : (i,j) \in S \}, \quad R_i' = \{ j \in \{1,\ldots,b\} : (i,j) \not\in S \}\\
    C_j &= \{ i \in \{1,\ldots,b\} : (i,j) \in S \}, \quad C_j' = \{ i \in \{1,\ldots,b\} : (i,j) \not\in S \}.
\end{align*}
$R_i$ is the set of all indices $j$ so that $\bm{D}_{i,j} \neq \bm{0}$ and $R_i'$ is the set of all indices $j$ so that $\bm{D}_{i,j} = \bm{0}$. Similarly, $C_j$ is the set of all indices $i$ so that $\bm{D}_{i,j} \neq \bm{0}$ and $C_i'$ is the set of all indices $i$ so that $\bm{D}_{i,j} = \bm{0}$. In the language of hierarchical partitions, $R_i'$ and $C_j'$ correspond to all admissible pairs of index sets for  a fixed row index set $i$ or fixed column index set $j$, respectively. 

Consider a matrix $\bm{B} \in \mathbb{R}^{bm \times bm}$ partitioned as
\begin{equation*}
    \bm{B} = \begin{bmatrix}
        \vec{B}_{1,1} & \cdots & \vec{B}_{1,b} \\
        \vdots & \ddots &\vdots \\
        \vec{B}_{b,1} & \cdots & \vec{B}_{b,b} \\
    \end{bmatrix}, \quad \bm{B}_{i,j} \in \mathbb{R}^{m \times m}.
\end{equation*}
For any fixed $i,j \in \{1,\ldots,b\}$ assume 
\begin{equation*}
    R_i' = \{(i,j_1),\ldots,(i,j_{r_i})\}, \quad C_j' = \{(i_1,j),\ldots,(i_{c_j},j)\}
\end{equation*}
and define
\begin{equation*}
    \gamma_j(\bm{B}) = \begin{bmatrix} \bm{B}_{i_1,j} \\ \vdots \\ \bm{B}_{i_{c_j},j}\end{bmatrix} \quad \text{and} \quad \rho_i(\bm{B}) = \begin{bmatrix} \bm{B}_{i,j_1} & \cdots & \bm{B}_{i,j_{r_i}}\end{bmatrix}.
\end{equation*}
Note that in the case when $\UBLR(b,m,k,S) = \SSS(\ell,k)$, $\gamma_j(\bm{B})$ and $\rho_i(\bm{B})$ correspond to the block-columns and block-rows introduced in \Cref{def:HSScolrow}. Furthermore, if $\bm{B} \in \UBLR(b,m,k,S)$, then for all $1\leq i,j\leq b$, we have that $\rank(\gamma_j(\bm{B})), \rank(\rho_i(\bm{B})) \leq k$. 

Let $s_{\max}$ be the maximum number of non-zero blocks per block-row and block-column of $\bm{D}$, i.e.
\begin{equation*}
    s_{\max} := \max\limits_{i,j \in \{1,\ldots,b\}} \left\{ |C_j|, |R_i|\right\},
\end{equation*}
where $|C_i|$ and $|R_j|$ denotes the cardinality of the sets $C_i$ and $R_j$, respectively. Hence, in the case when $\UBLR(b,m,k,S) = \SSS(\ell,k)$ we have $s_{\max} = 1$. 
    
\subsubsection{Preliminary results}
In this section, we present an analogue of \Cref{thm:simple_error} for $\UBLR(b,m,k,S)$.
\begin{theorem}\label{thm:simple_error_ublr}
    Suppose that for all $1 \leq i,j \leq b$,  $\bm{U}_i$ and $\bm{V}_i$ satisfy
    \begin{align*}
        \EE\Bigl[ \|\rho_i(\bm{A}) - \bm{U}_i \bm{U}_i^\T \rho_i(\bm{A})\|_\F^2 \Big]
        &\leq \Gamma_{\textup{r}} \cdot \|\rho_i(\bm{A}) - \llbracket \rho_i(\bm{A}) \rrbracket_k\|_\F^2 
        \\
        \EE\Bigl[ \| \gamma_j(\bm{A}) - \gamma_j(\bm{A}) \bm{V}_j \bm{V}_j^\T\|_\F^2 \Big]
        &\leq \Gamma_{\textup{c}} \cdot \|\gamma_j(\bm{A}) - \llbracket \gamma_j(\bm{A}) \rrbracket_k\|_\F^2,
    \end{align*}
    and for all $(i,j) \in S$, $\bm{D}_{i,j}$ satisfies  
    \begin{align*}
        &
        \EE\Bigl[ \| (\bm{A}_{i,j} - \bm{D}_{i,j}) - \bm{U}_i \bm{U}_i^\T (\bm{A}_{i,j} - \bm{D}_{i,j})\bm{V}_i \bm{V}_i^\T \|_\F^2 \Big| \bm{U}_i,\bm{V}_i \Big]
        \\&\hspace{5em}\leq \Gamma_{\textup{d}} \cdot \left(\|\rho_i(\bm{A}) - \bm{U}_i \bm{U}_i^\T \rho_i(\bm{A})\|_\F^2 +  \| \gamma_j(\bm{A}) - \gamma_j(\bm{A}) \bm{V}_j \bm{V}_j^\T\|_\F^2 \right).
    \end{align*}

    Let $\bm{U} = \blockdiag\left(\bm{U}_1,\ldots,\bm{U}_{b}\right)$, $\bm{V} = \blockdiag\left(\bm{V}_1,\ldots,\bm{V}_{b}\right)$,
    \begin{equation}
        \bm{D} = \begin{bmatrix}
        \vec{D}_{1,1} & \cdots & \vec{D}_{1,b} \\
        \vdots & \ddots &\vdots \\
        \vec{D}_{b,1} & \cdots & \vec{D}_{b,b} \\
    \end{bmatrix}, \quad \text{where } \bm{D}_{i,j} = \bm{0} \text{ for } (i,j) \not\in S
    \end{equation}
    and 
    $\bm{X} = \bm{U}^\T(\bm{A} - \bm{D}) \bm{V}$. 
    Then, 
    \[
    \EE\Bigl[ \| \bm{A} - (\bm{U}\bm{X}\bm{V}^\T + \bm{D}) \|_\F^2 \Bigr] \leq (\Gamma_{\textup{r}}+\Gamma_{\textup{c}})(1+\Gamma_{\textup{d}}) \cdot \rev{\min\limits_{\bm{B} \in \UBLR(b,m,k,S)}}\|\bm{A}-\bm{B}\|_\F^2.
    \]
\end{theorem}
Noting that the error satisfies
\begin{equation*}
    \| \bm{A} - (\bm{U}\bm{X}\bm{V}^\T + \bm{D}) \|_\F^2 = \sum\limits_{(i,j) \in S} \|\bm{A}_{i,j} - \bm{U}_i \bm{X}_{i,j} \bm{V}_j^\T - \bm{D}_{i,j}\|_\F^2 + \sum\limits_{(i,j) \not\in S} \|\bm{A}_{i,j} - \bm{U}_i \bm{X}_{i,j} \bm{V}_j^\T\|_\F^2,
\end{equation*}
we can prove \Cref{thm:simple_error_ublr} using exactly the same arguments as for the proof \Cref{thm:simple_error}.

\subsubsection{Block nullification for uniform block low-rank matrices}\label{section:blocknullification_ublr}
In \Cref{section:blocknullification},  we used block nullification to compute sketches of the block-rows and block-columns. The authors of \cite{PearceYesypenkoLevittMartinsson:2025} modify this idea to compute sketches with $R_i(\bm{A})$ and $C_j(\bm{A})$ presented in \Cref{section:ublr_notation}. This section is devoted to describing this idea.

Consider a matrix $\bm{A}$ and $\bm{\Omega} \sim \Gaussian(bm, s)$ partitioned as
\begin{equation}\label{eq:ublr_omega_partition}
    \bm{\Omega} = \begin{bmatrix} \bm{\Omega}_1 \\ \bm{\Omega}_2 \\ \vdots \\ \bm{\Omega}_b \end{bmatrix}, \quad \bm{\Omega}_i \in \mathbb{R}^{m \times s}.
\end{equation}
Recalling that $R_i' = \{(i,j_1),\ldots,(i,j_{r_i})\}$ define
\begin{equation*}
    \bm{\Omega}_{R_i'} = \begin{bmatrix} \bm{\Omega}_{j_1} \\ \vdots \\ \bm{\Omega}_{r_i} \end{bmatrix}
\end{equation*}
and define $\bm{\Omega}_{R_i}$ to be the matrix containing the concatenated blocks of $\bm{\Omega}$ not contained in $\bm{\Omega}_{R_i'}$. Note that $\bm{\Omega}_{R_i}$ has at most $s_{\max} m$ rows and therefore has a nullspace of dimension at least $s - s_{\max}m$. Therefore, we can obtain a matrix $\bm{P}_{R_i}$ whose columns form an orthonormal basis for the nullspace so that $\bm{\Omega}_{R_i}\bm{P}_{R_i} = 0$. Moreover, by the unitary invariance of Gaussian matrices together with the independence of the different blocks of $\bm{\Omega}$,  $\bm{G}_{R_i} :=\bm{\Omega}_{R_i'} \bm{P}_{R_i} \sim \Gaussian(r_i m,s-s_{\max}m)$. Furthermore, note that if $\bm{Y} = \bm{A} \bm{\Omega}$, then
\begin{equation*}
    \bm{Y}_i \bm{P}_{R_i} = \rho_i(\bm{A}) \bm{G}_{R_i}.
\end{equation*}
Thus, we can implicitly compute Gaussian sketches of $\rho_i(\bm{A})$ from the sketch $\bm{A} \bm{\Omega}$. By an identical argument, we can implicitly compute Gaussian sketches of $\gamma_i(\bm{A})$ from the sketch $\bm{Z} = \bm{A}^\T \bm{\Psi}$, where $\bm{\Psi} \sim \Gaussian(bm, s)$, by computing $\bm{Z}_j \bm{Q}_{C_j}$ where $\bm{Q}_{C_j}$ is an orthonormal basis for the nullspace of $\bm{\Psi}_{C_j}$.

By the discussion above together with \Cref{thm:range_finder} we  get the following result. 

\begin{theorem}\label{thm:range_finder_ublr}
    Let $\bm{A}\in \mathbb{R}^{bm\times bm}$ and suppose $\bm{\Omega}, \bm{\Psi}\sim\operatorname{Gaussian}(bm,s)$ where $s \geq s_{\max} m + k + 2$. 
    Let $\bm{U}_i$ be the top $k$ left singular vectors of $\bm{A}\bm{\Omega} \bm{P}_{R_i}$ and $\bm{V}_j$ be the top $k$ left singular vectors of $\bm{A}^\T \bm{\Psi} \bm{Q}_{C_j}$. Then,
    \[
    \EE\Bigl[ \| \rho_i(\bm{A}) - \bm{U}_i\bm{U}_i^\T \rho_i(\bm{A}) \|_\F^2 \Big]
    \leq \left(1 + \frac{2e(s-s_{\max}m)}{\sqrt{(s-s_{\max}m-k)^2-1}}\right)^2 \| \rho_i(\bm{A}) - \llbracket \rho_i(\bm{A}) \rrbracket_k \|_\F^2.
    \]
    and
    \[
    \EE\Bigl[ \| \gamma_j(\bm{A}) - \gamma_i(\bm{A}) \bm{V}_i\bm{V}_i^\T\|_\F^2 \Big]
    \leq \left(1 + \frac{2e(s-s_{\max}m)}{\sqrt{(s-s_{\max}m-k)^2-1}}\right)^2 \| \gamma_j(\bm{A}) - \llbracket \gamma_j(\bm{A}) \rrbracket_k \|_\F^2.
    \]
\end{theorem}

\subsubsection{Computing the block-sparse matrix}\label{section:diagonal_ublr}
The previous section provided a method to compute near-optimal bases $\bm{U}_1,\ldots,\bm{U}_b,\bm{V}_1,\ldots,\bm{V}_b$. In this section we show how we can compute a matrix $\bm{D}$ satisfying the condition in \Cref{thm:simple_error_ublr}. 

\begin{theorem}\label{theorem:diagonal_ublr}
    Consider a matrix $\bm{B} \in \mathbb{R}^{bm \times bm}$ as partitioned in \Cref{def:UBLR}. Fix some bases
    \begin{align*}
    \bm{U}_1,\ldots,\bm{U}_{b} \quad &\text{where } \bm{U}_i \in \mathbb{R}^{m \times k} \text{ has orthonormal columns,}\\
     \bm{V}_1,\ldots,\bm{V}_{b} \quad &\text{where } \bm{V}_i \in \mathbb{R}^{m \times k} \text{ has orthonormal columns}.
\end{align*}
Let $\bm{\Omega},\bm{\Psi} \sim \operatorname{Gaussian}(m,s)$, where $s \geq s_{\max}m + k + 2$, be independent and partition
\begin{equation}\label{eq:UBLR_Ypartition}
    \bm{\Omega} = \begin{bmatrix} \bm{\Omega}_1 \\ \bm{\Omega}_2 \\ \vdots \\ \bm{\Omega}_{b} \end{bmatrix}, \quad\bm{\Psi} = \begin{bmatrix} \bm{\Psi}_1 \\ \bm{\Psi}_2 \\ \vdots \\ \bm{\Psi}_{b} \end{bmatrix}, \quad  \bm{Y} = \bm{B} \bm{\Omega} = \begin{bmatrix} \bm{Y}_1 \\ \bm{Y}_2 \\ \vdots \\ \bm{Y}_{b} \end{bmatrix}, \quad \bm{Z} = \bm{B}^\T \bm{\Psi} = \begin{bmatrix} \bm{Z}_1 \\ \bm{Z}_2 \\ \vdots \\ \bm{Z}_{b} \end{bmatrix},
\end{equation}
where $\bm{\Omega}_i, \bm{\Psi}_i,\bm{Y}_i, \bm{Z}_i \in \mathbb{R}^{m \times s}$. Furthermore, let $\bm{\Omega}_{R_i}$ and $\bm{\Psi}_{C_j}$ be defined as in \Cref{section:blocknullification_ublr}. For $(i,j) \in S$ define
\begin{equation*}
    \bm{D}_{i,j} = (\bm{I}-\bm{U}_i \bm{U}_i^\T) \bm{Y}_i \bm{\Omega}_{R_i}^{\dagger} + \bm{U}_i \bm{U}_i((\bm{I} - \bm{V}_j \bm{V}_j^\T) \bm{Z}_{j} \bm{\Psi}_{C_j}^{\dagger})^\T.
\end{equation*}
Then, for each $(i,j) \in S$ we have 
\begin{align*}
    &\mathbb{E}\Bigl[\|\bm{B}_{i,j} - \bm{D}_i - \bm{U}_i \bm{U}_i^\T(\bm{B}_{ii} - \bm{D}_i)\bm{V}_i \bm{V}_i^\T\|_\F^2 \Bigr] 
    \\&\hspace{5em}\leq \frac{s_{\max}m}{s-s_{\max}m - 1} \left(\|(\bm{I} - \bm{U}_i \bm{U}_i^\T) \rho_i(\bm{B})\|_\F^2 + \| \gamma_i(\bm{B})(\bm{I} - \bm{V}_i \bm{V}_i^\T)\|_\F^2 \right)
\end{align*}
\end{theorem}
The proof is identical to \Cref{theorem:diagonal}, which is a special case of \Cref{theorem:diagonal_ublr} when $S = \{(i,i):i = 1,\ldots,b\}$.

\subsubsection{The algorithm}
Using the discussion from \Cref{section:blocknullification_ublr} and \Cref{section:diagonal_ublr} we can now present an algorithm to compute quasi-optimal approximation $\bm{B} \in \UBLR(b,m,k,S)$ to $\bm{A}$ using only matrix-vector products with $\bm{A}$ and $\bm{A}^\T$; see \Cref{alg:UBLR_greedy_matvec}.
\begin{algorithm}
\caption{Greedy $\UBLR$ approximation from matvecs}
\label{alg:UBLR_greedy_matvec}
\textbf{input:} A matrix $\bm{A} \in \mathbb{R}^{N \times N}$ implicitly given through matrix-vector products. A number of columns of the random matrices $s \geq s_{max} m + k + 2$.\\
\textbf{output:} An approximation $\bm{B} \in \UBLR(b,m,k,S)$ to $\bm{A}$.
\begin{algorithmic}[1]
    \State Sample $\bm \Omega, \widetilde{\bm \Omega}, \bm  \Psi, \widetilde{\bm{\Psi}} \sim \operatorname{Gaussian}(bm, s)$ partitioned as per \Cref{eq:ublr_omega_partition}.
    \State Compute the following  partitioned  as per \Cref{eq:UBLR_Ypartition}: \begin{align*}& \bm Y= \bm{A}\bm \Omega \qquad  &\widetilde{\bm{Y}} = \bm{A} \widetilde{\bm{\Omega}} \\  &\bm Z =  \bm{A}^{\T} \bm \Psi \qquad &\widetilde{\bm{Z}}=\bm{A}^{\T}  \widetilde{\bm{\Psi}}\end{align*}

    \For{row blocks $i = 1, \dots, b$ and column blocks $j = 1, \dots, b$ }
    \State Let $\bm{P}_{R_i}$ be an orthonormal basis of $\text{nullspace}(\bm{\Omega}_{R_i})$.
    \State Let $\bm U_i \in \mathbb{R}^{m \times k}$ be the top $k$ left singular vectors of $\bm{Y}_{R_j} \bm{P}_{C_j}$.
    \State Let $\bm{Q}_{C_j}$ be an orthonormal basis of $\text{nullspace}(\bm{\Psi}_{C_j})$.
    \State Let $\bm{V}_j\in \mathbb{R}^{m \times k}$ be the top $k$ left singular vectors of $\bm{Z}_{C_j} \bm{Q}_{R_i}$.
    
   \EndFor
   \State $\bm U = \text{blockdiag} (\bm{U}_1,\ldots,\bm{U}_{b})$
   \State $\bm V = \text{blockdiag} (\bm{V}_1,\ldots,\bm{V}_{b})$
   \For{$(i, j) \in S$} 
   \State $\bm D_{i,j} = (\bm I - \bm U_i (\bm U_i)^\T )\widetilde{\bm{Y}}_i (\widetilde{\bm{\Omega}}_{R_i}^\dagger + \bm U_i(\bm U_i)^\T [(\bm I - \bm V_j (\bm V_j)^\T) \widetilde{\bm{Z}}_j(\widetilde{\bm{\Psi}} _{C_j})^\dagger]^\T $
   \EndFor
   \State $\bm X = (\bm{U})^\T \left(\bm{A} - \bm{D}\right)\bm{V}$ 
\end{algorithmic}
\end{algorithm}

From \Cref{thm:simple_error_ublr}, \Cref{thm:range_finder_ublr}, and \Cref{theorem:diagonal_ublr} we immediately get the following result for the output of \Cref{alg:UBLR_greedy_matvec}.

\begin{theorem}\label{theorem:ublr_matvec}
    If $s \geq s_{\max}m + k + 2$, \Cref{alg:UBLR_greedy_matvec} returns an approximation $\bm{B} \in \UBLR(b,m,k,S)$ to $\bm{A}$ so that \Cref{thm:simple_error_ublr} holds with
    \begin{equation*}
        \Gamma_{\textup{r}} = \Gamma_{\textup{c}} = \left(1 + \frac{2e(s-s_{\max}m)}{\sqrt{(s-s_{\max}m-k)^2-1}}\right)^2,
        \qquad
        \Gamma_{\textup{d}} = \frac{s_{\max}m}{s-s_{\max}m -1}.
    \end{equation*}
    Hence, with $O(s_{\max}m + k)$ matrix-vector products with $\bm{A}$ and $\bm{A}^\T$ \Cref{alg:UBLR_greedy_matvec} returns an approximation to $\bm{A}$ so that
    \begin{equation*}
        \mathbb{E}\Bigl[ \|\bm{A} - \bm{B}\|_\F^2 \Bigr] \leq \Gamma\cdot \rev{\min\limits_{\bm{C} \in \UBLR(b,m,k,S)}}\|\bm{A}-\bm{C}\|_\F^2
        , \quad \text{where} \quad \Gamma = O(1).
    \end{equation*}
\end{theorem}
\section{\rev{The existence of an optimal HSS approximation}}\label{appendix:existence}
\rev{Unlike HODLR matrices, where the different levels of the hierarchy are independent, HSS matrices impose \emph{nestedness} conditions across different levels. The nestedness conditions make it less clear why an \emph{optimal} approximation from the set $\HSS(L,k)$ should exist. We are not aware of any previous results that guarantee the existence of an optimal HSS approximation. Therefore, for completeness, we include a proof here. }
\begin{theorem}
    \rev{For any $\bm{A} \in \mathbb{R}^{2^{L+1}k \times 2^{L+1}k}$, there exists $\widetilde{\bm{B}} \in \HSS(L,k)$ so that $$\|\bm{A} - \widetilde{\bm{B}}\|_\F = \inf\limits_{\bm{B} \in \HSS(L,k)}\|\bm{A} - \bm{B}\|_\F.$$}
\end{theorem}
\begin{proof}
    \rev{The function $g(\bm{B}):= \|\bm{A} - \bm{B}\|_\F$ is a continuous and coercive function. To ensure the existence of a minimum it is therefore sufficient to prove that the set $\HSS(L,k)$ is closed \cite[Proposition A.8]{bertsekas1997nonlinear}. }
    
    \rev{Any $\bm{B} \in \HSS(L,k)$ admits a telescoping factorization with telescoping rank $k$. By \cite[Section 3.3]{CasulliKressnerRobol:2024} any matrix that admits a telescoping factorization with telescoping rank at most $k$ is also an HSS matrix with HSS rank at most $k$, and by \cite[Section 2.2]{MasseiRobolKressner:2020} a matrix is HSS if and only if all HSS block columns and rows, at every level, have rank at most $k$.}
    
    \rev{Enumerate the HSS block columns and rows from $1$ to $S$, where $S$ is the total number of HSS block columns and rows. Define the function $f: \mathbb{R}^{N \times N} \to \mathbb{R}^S_{\geq 0}$ so that $f(\bm{B})_i$ equals the $(k+1)^{\text{th}}$ singular value of the $i^{\text{th}}$ HSS block column/row. Note that by the stability of the singular values \cite[Problem 111.6.13]{bhatia}, $f$ is a vector-valued function whose components are all continuous, so $f$ is also continuous. }

    \rev{By the characterization above we have
    \begin{equation*}
        \HSS(L,k) = \{\bm{B} \in \mathbb{R}^{2^{L+1}k \times 2^{L+1}k} : f(\bm{A}) = \bm{0}\},
    \end{equation*}
    which is a closed set by the continuity of $f$.}
\end{proof}

\end{document}